\documentclass[reqno, 11pt]{amsart}

\usepackage[parfill]{parskip}    
\usepackage{amsrefs, amsmath}
\usepackage{diagbox}
\usepackage{amsfonts, mathrsfs}
\usepackage{amssymb}
\usepackage{amscd}
\usepackage{cleveref, multirow} 
\usepackage{fullpage, color}
\usepackage[labelformat=empty]{subfig}
\usepackage{booktabs}
\usepackage{xypic}
\usepackage[dvips]{graphicx,epsfig}
\usepackage{enumerate, longtable}
\usepackage{makecell}

\numberwithin{equation}{section}

\BibSpec{article}{%
+{}{\sc \PrintAuthors} {author}
+{,}{ \textrm} {title}
+{,}{ \textit} {journal}
+{,}{ } {volume}
+{}{ \IfEmptyBibField{journal}{}{\PrintDate{date}}} {transition}
+{,}{ } {pages}
+{,}{ } {status}
+{.}{} {transition}
+{}{ Preprint available at \tt} {eprint}
+{.}{} {transition}
}

\BibSpec{book}{%
+{}{\sc \PrintAuthors} {author}
+{,}{ \textit} {title}
+{.}{ } {series}
+{,}{ } {volume}
+{.}{ } {publisher}
+{,}{ } {place}
+{,}{ } {date}
+{.}{} {transition}
}

\BibSpec{collection.article}{
+{} {\sc \PrintAuthors} {author}
+{,} { \textrm } {title}
+{,} { in \it \PrintConference} {conference}
+{,} { pp.~} {pages}
+{.} {\PrintBook} {book}
+{,} { \PrintDateB} {date}
+{.} {} {transition}
}

\BibSpec{innerbook}{
+{.} { \emph} {title}
+{.} { } {part}
+{:} { \emph} {subtitle}
+{.} { } {series}
+{,} { } {volume}
+{.} { Edited by \PrintNameList} {editor}
+{.} { Translated by \PrintNameList}{translator}
+{.} { \PrintContributions} {contribution}
+{.} { } {publisher}
+{.} { } {organization}
+{,} { } {address}
+{,} { \PrintEdition} {edition}
+{,} { \PrintDateB} {date}
+{.} { } {note}
}

\allowdisplaybreaks[3]

\crefname{equation}{Eq.}{Eqs.}
\crefname{eqnarray}{Eq.}{Eqs.}
\crefname{algo}{Algorithm}{Algorithms}
\crefname{conj}{Conjecture}{Conjectures}
\crefname{defn}{Definition}{Definitions}
\crefname{lem}{Lemma}{Lemmas}
\crefname{thm}{Theorem}{Theorems}
\crefname{claim}{Claim}{Claims}
\crefname{rmk}{Remark}{Remarks}
\crefname{prop}{Proposition}{Propositions}
\crefname{section}{Section}{Sections}
\crefname{appendix}{Appendix}{Appendices}
\crefname{cor}{Corollary}{Corollaries}
\crefname{figure}{Figure}{Figures}
\crefname{table}{Table}{Tables}
\crefname{example}{Example}{Examples}
\crefname{prob}{Problem}{Problems}
\crefname{assm}{Assumption}{Assumptions}

\newcommand{\de}{{\partial}}

\newcommand{\rd}{\mathrm{d}}
\newcommand{\ri}{\mathrm{i}}
\newcommand{\re}{\mathrm{e}}

\newcommand{\bbV}{\mathbb{V}}

\newcommand{\bbA}{\mathbb{A}}

\newcommand{\bbN}{\mathbb{N}}

\newcommand{\bbZ}{\mathbb{Z}}
\newcommand{\bbR}{\mathbb{R}}
\newcommand{\bbC}{\mathbb{C}}
\newcommand{\bbP}{\mathbb{P}}

\newcommand{\bbQ}{\mathbb{Q}}

\def\bary{\begin{array}} 
\def\eary{\end{array}} 
\def\ben{\begin{enumerate}} 
\def\een{\end{enumerate}}
\def\bit{\begin{itemize}} 
\def\eit{\end{itemize}}
\def\nn{\nonumber} 

\newcommand{\cO}{\mathcal{O}}
\newcommand{\cT}{\mathcal{T}}

\newcommand{\cP}{\mathcal{P}}

\newcommand{\DD}{\mathcal{D}}
\newcommand{\LL}{\mathcal{L}}

\newcommand{\cN}{\mathcal{N}}
\newcommand{\cW}{\mathcal{W}}
\newcommand{\cG}{\mathcal{G}}
\newcommand{\cA}{\mathcal{A}}

\newcommand{\cB}{\mathcal{B}}

\newcommand{\cJ}{\mathcal{J}}

\newcommand{\cM}{\mathcal M}

\def\beq{\begin{equation}}                     %
\def\eeq{\end{equation}}                       %
\def\bea{\begin{eqnarray}}                     
\def\eea{\end{eqnarray}}
\def\bary{\begin{array}} 
\def\eary{\end{array}} 
\def\ben{\begin{enumerate}} 
\def\een{\end{enumerate}}
\def\bit{\begin{itemize}} 
\def\eit{\end{itemize}}
\def\nn{\nonumber} 
\def\de {\partial}

\def\a{\alpha}

\theoremstyle{plain}

\newtheorem{thm}{Theorem}[section]
\newtheorem{lem}[thm]{Lemma}
\newtheorem{algo}{Algorithm}[section]
\newtheorem{assm}{Assumption}[section]
\newtheorem{prop}[thm]{Proposition}

\newtheorem*{conj*}{Conjecture}
\newtheorem{cor}[thm]{Corollary}
\newtheorem*{cor*}{Corollary}

\newtheorem{prob}{Problem}[section]
\newtheorem{defn}{Definition}[section]

\theoremstyle{definition}

\newtheorem{rmk}[thm]{Remark}
\newtheorem{example}{Example}[section]

\newcommand{\GITl}[1]{\backslash \!\! \backslash _{\kern-.2em #1 \kern0.1em}}
\newcommand{\GIT}[1]{/\!\!/_{\kern-.2em #1 \kern0.1em}}

\renewcommand{\l}{\left}
\renewcommand{\r}{\right}

\def\bred{\begin{color}{red}}
\def\ered{\end{color}}

\def\bes{\begin{subequations}}
\def\ees{\end{subequations}}

\begin{document}

\title{Exterior powers of the adjoint representation and the
  Weyl ring of $\mathrm{E}_8$}

\author{Andrea Brini}
\address{School of Mathematics,
Edgbaston, B15 2TT, Birmingham, United Kingdom
}
\address{School of Mathematics and Statistics,
Hounsfield Road, S3 7RH, Sheffield, United Kingdom
}
\address{On leave from IMAG, Univ.~Montpellier, CNRS, Montpellier, France}
\email{a.brini@sheffield.ac.uk}

\begin{abstract}

  I derive explicitly all polynomial relations in the
  character ring of $E_8$ of the form $\chi_{\wedge^k \mathfrak{e}_8} -
  \mathfrak{p}_{k} (\chi_{1}, \dots, \chi_{8})=0$, where
  $\wedge^k \mathfrak{e}_8$ is an arbitrary exterior power of the adjoint
  representation and $\chi_{i}$ is the $i^{\rm th}$ fundamental character. This has
  simultaneous implications for the theory of relativistic integrable systems,
  Seiberg--Witten theory, quantum topology, orbifold Gromov--Witten theory,
  and the arithmetic of elliptic curves. The solution is obtained by reducing the problem to
  a (large, but finite) dimensional linear problem, which is amenable to an
  efficient solution via distributed computation.

\end{abstract}

\maketitle
\tableofcontents

\section{Introduction}

\subsection{The problem}

Let $\cG$ be a complex, simple, simply-connected Lie group of rank $r$ and $V \in \mathrm{Rep}_\bbC(\cG)$ an
element of its representation ring. We may view the latter, upon taking characters $V \to
\chi_V$, as the Weyl ring of Ad-invariant regular functions on $\cG$, or
equivalently, as the ring of $\cW(\cG)$-invariant regular functions
on the Cartan torus, where $\cW(\cG)$ is the Weyl group of $\cG$. It is a basic fact in Lie theory that this is a polyomial ring
over the integers, 
%
$\mathrm{Rep}_\bbC(\cG) \simeq \bbZ[\chi_{1}, \dots, \chi_{r}]$,
%
where $\chi_{j}$, $j=1, \dots, r$ denotes the character of
the $j^{\rm th}$ fundamental representation of
$\cG$. 

In this paper we will be concerned with a special instance of the following
\begin{prob}
Given a finite-dimensional representation $V$ of $\cG$, find polynomials
$\mathfrak{p}_k^V\in \bbZ[x_1, \dots, x_r]$, $k=1, \dots, \mathrm{dim}_\bbC V$, such that
\beq
\det_V(g - \mu \mathbf{1})=\sum_k \mathfrak{p}_k^V(\chi_{1}(g), \dots,
\chi_{r}(g)) (-\mu)^{\mathrm{dim}V-k} \in \bbZ[\chi_{1}(g), \dots,
  \chi_{r}(g)][\mu].
\label{eq:cpol}
\eeq
for all $g\in\cG$ and $k=1,\dots,\mathrm{dim}_\bbC V$. Equivalently, given an arbitrary exterior power of $V$, determine 
the corresponding polynomial relations in $\mathrm{Rep}(\cG)$ of the form
\beq
0=
\chi_{\wedge^k V} - \mathfrak{p}_{k}^V (\chi_{1}, \dots, \chi_{r}).
\eeq
\label{prob:gen}
\end{prob}

In other words, \cref{prob:gen} asks to find explicit expressions for characteristic
polynomials in a given representation $V$ (alternatively, of antisymmetric
characters of $V$) in terms of polynomials in the
fundamental characters.
For example, if  $\cG=\mathrm{SL}_N(\bbC)$ and $V=V_{\omega_1}$ is the defining
representation of $\mathrm{SL}_N(\bbC)$,
we have simply
\beq
\mathfrak{p}_k^V(\chi_{1}, \dots,
\chi_{r}) = \chi_{k}
\eeq
since
$V_{\omega_k}=\wedge^k V_{\omega_1}$.

Let $\mathfrak{d}_0(V)$ denote the dimension of the zero-weight space of
$V$. A case of particular importance for applications is when 
the characteristic polynomial $\det_V(g - \mu \mathbf{1})$
(respectively: $\det_{V}(g - \mu
\mathbf{1})(1-\mu)^{-\mathfrak{d}_0(V)}$) in \eqref{eq:cpol} is irreducible over
$\mathrm{Rep}(\cG)$: this amounts to $V$ being a minuscule (respectively:
quasi-minuscule) irreducible representation. In the quasi-minuscule setting, \cref{prob:gen} is computationally
easy for most Dynkin types and quasi-minuscule representations (see \cite{Borot:2015fxa, Brini:2017gfi}), with one
single, egregious exception: this is $\cG=E_8$ and $V=\mathfrak{e}_8$,
which is of
formidable complexity. The purpose of the present paper is 
to present a solution of this
exceptional case, which had previously been announced in
\cite[Appendix~A]{Brini:2017gfi}.

%
\subsection{Motivation: six places of appearance of \cref{prob:gen}}

As it stands, \cref{prob:gen} is of purely representation theoretic character. At the same
time, my motivation for looking at it is mostly extrinsic in nature, and
it is eminently geometrical: there are indeed six different classes of
questions in Geometry and Mathematical Physics that are
simultaneously answered by giving an explicit solution to \cref{prob:gen}, and
in particular to the case when $(\cG, V)=(E_8, \mathfrak{e}_8)$, as follows.
\subsubsection{Integrable systems, spectral curves, and the relativistic Toda chain}
\label{sec:is}
Let 
$\cT$ denote the maximal torus of $\cG$. A central object in the theory of algebraically complete integrable
systems is the datum of a rational map
\beq
\LL : C_{\mathsf{g}} \times \cP \to \cG
\eeq
from the product of a $2r$-complex algebraic symplectic variety $\cP$ and a fixed smooth curve $C_{\mathsf{g}}$ with $h^{1,0}(C_{\mathsf{g}})=\mathsf{g}$ to $\cG$, called
  the {\it Lax map}. Associated to $\LL$ is the family of {\it spectral curves} $\mathscr{C}[\LL]$
%
\beq
\xymatrix{ \Sigma_{u} \ar[d]  \ar@{^{(}->}[r]&
  \mathscr{C}[\LL] \ar[d]  \\
                     u  \ar@{^{(}->}[r]^{\rm pt} 
 &  
\mathscr{B} & 
}
\label{eq:diagsc}
\eeq
where
\bit
\item $\mathscr{B} = \chi_{\bullet}(\cT)$ is the image of the torus
  under the fundamental regular characters $\chi_{i}$, $i=1, \dots, r$;
  the natural co-ordinates $u$ on $\mathscr{B}$, $u_i \triangleq
  \chi_{i}(g)$ for $g \in \cG$, give a co-ordinate chart on $\cT$;
\item for fixed $u=(u_i=\chi_{i}(\LL))_i$, $\Sigma_u$ is the compact Riemann surface
\beq
\Sigma_u = \overline{\bbV\l(\det_V(\LL-\mu \mathbf{1})\r)}
\eeq
 given by the smooth completion (normalisation of the projective closure) of
 the algebraic curve in $\bbA^1 \times C_{\mathsf{g}}$ given by the vanishing locus of
 the characteristic polynomial of $\LL$ at fixed $u$. 
\eit
More explicitly,
\beq
\det_V(\LL-\mu \mathbf{1}) = \sum_{k=0}^{\dim_V}
(-\mu)^{\dim_V -k} \chi_{\wedge^k V} \LL \in \cM_{C_{\mathsf{g}}}[\mu; u_1,
  \dots, u_r]
\label{eq:specdet}
\eeq
where $\cM_{C_{\mathsf{g}}}$ is the ring of meromorphic functions on $C_{\mathsf{g}}$. A Hamiltonian
dynamical system can then be defined on $\cP$ in terms of isospectral flows on $\LL$;
in particular, any Ad-invariant function of $\LL$ is an integral of motion, and a
maximal involutive set of these is given by the fundamental characters
$u_i(\LL)$. For fixed $u$, the resulting flow is linear on the Picard group of
the irreducible components
$\Sigma_u$, and the dynamics is independent of $V$ so long as $V \neq
\mathbf{1}$ \cite{MR1397059, MR1013158, MR1401779, MR1668594}: so for simplicity, we may
assume $V$ to be minuscule (resp. quasi-minuscule for $\cG=E_8$), which entails
that $\Sigma_u$ (resp. the reduced component of $\Sigma_u$) is generically
irreducible.

A central example is given by the periodic relativistic Toda
chain of type $\cG$: in this case $C_{\mathsf{g}}\simeq \bbP^1$, and
in an affine co-ordinate $\lambda$ on
the base $\bbP^1$ the Lax map satisfies \cite{Borot:2015fxa,Brini:2017gfi}
\beq
\de_\lambda
u_i(\LL)= \delta_{i i_{\rm top}} \l(1-\frac{1}{\lambda^2}\r),
\label{eq:uilambda}
\eeq
 where
$V_{i_{\rm top}}$ is the top dimensional fundamental representation. This
 means that
\eqref{eq:specdet}
becomes
\beq
\det_V(\LL-\mu \mathbf{1}) = \sum_{k=0}^{\dim V}
(-\mu)^{\dim V -k} \chi_{\wedge^kV} \LL \in \bbZ[\mu, \lambda^{\pm}; u_1,
  \dots, u_r]
\eeq
where the last step requires expanding $\chi_{\wedge^kV}$ as a polynomial
in $\chi_{i}$ from the solution of \cref{prob:gen}, and using \eqref{eq:uilambda}. A complete presentation for the family of
spectral curves \eqref{eq:diagsc}, and the complete solution for the dynamics of
the underlying integrable model, follows thus from solving \cref{prob:gen} for
the given pair $(\cG,V)$.

\subsubsection{Gauge theory I : Seiberg--Witten theory}
\label{sec:sw}
The Toda systems of the previous section have a central place in the study of
supersymmetric quantum field theories \cite{Gorsky:1995zq,Martinec:1995by,Nekrasov:1996cz}.
In particular, constructing explicitly the family of spectral curves \eqref{eq:diagsc}
encodes the solution of the low energy effective dynamics for $\cN=1$
supersymmetric gauge theories with no hypermultiplets on $\bbR^{1,3} \times S^1$: in this setting, the
fundamental characters $u_i$ are the semiclassical, Weyl-invariant gauge
parameters and the full effective
action up to two derivatives in the supercurvature of the gauge field,
including all-order instanton corrections, is recovered via period integrals
of $\log \mu \rd \log \lambda$ on $\Sigma_u$. I refer the the reader to \cite{Borot:2015fxa, Brini:2017gfi, Nekrasov:1996cz} for a fuller discussion of the link between the
relativistic Toda chain and this class of five-dimensional quantum field
theories.

The case $\cG=E_8$ has been
outstanding since the solution of the Wilsonian dynamics of the theory was
proposed shortly after the celebrated work of Seiberg--Witten
\cite{Seiberg:1994rs, Martinec:1995by, Nekrasov:1996cz}; the same considerations for the relativistic
Toda chain recast this problem into solving \cref{prob:gen} for $(E_8,
\mathfrak{e}_8)$. 

\subsubsection{Gauge theory II: Chern--Simons theory}
\label{sec:cs}
Let $\mathfrak{g}$ be a simple complex Lie
algebra $\mathfrak{g}$ of type $A_n$, $D_n$ or $E_n$.
Let $V_{\mathfrak{g}}$ be the defining
representation for $\mathfrak{g}=\mathfrak{sl}(n+1)$, the vector representation for
$\mathfrak{g}=\mathfrak{so}(2n)$, and the $\mathbf{27}_{E_6}$, $\mathbf{56}_{E_7}$ and
$\mathbf{248}_{E_8}= \mathfrak{e}_8$ for $\mathfrak{g}=\mathfrak{e}_{6,7,8}$
respectively, and  denote by $\Gamma_{\mathfrak{g}}
< SU(2)$ the finite order subgroup of $SU(2)$ of the same Dynkin type of $\mathfrak{g}$
associated to 
$\mathfrak{g}$ by the classical McKay correspondence. Considerations about
large $N$ duality in gauge theory led \cite{Borot:2015fxa, Brini:2017gfi} to
propose that the family of curves 
\eqref{eq:diagsc} for $(\cG, V)=(\exp(\mathfrak{g}),V_{\mathfrak{g}})$ determines  the asymptotics of the
Witten--Reshetikhin--Turaev invariant
of spherical 3-space forms $M_{\mathfrak{g}} = S^3/\Gamma_{\mathfrak{g}}$ via the Chekhov--Eynard--Orantin topological
recursion \cite{Eynard:2004mh, Chekhov:2005rr, Eynard:2007kz}.
In
particular, the case $\Gamma_{\mathfrak{e}_8}=\mathbb{I}_{120}$ being the binary icosahedral
group gives the distinguished case of the Poincar\'e
integral homology sphere $M_{\mathfrak{e}_8} = \Sigma(2,3,5)$. The all-order
asymptotic expansion of the quantum invariants of the Poincar\'e sphere is
fully determined by periods of $\log\mu \rd\log \lambda$ and the topological
recursion on the spectral curves \eqref{eq:diagsc} for the
$\cG$-relativistic Toda chain  -- and therefore, ultimately, by solving \cref{prob:gen} for
$(\cG,V)=(E_8, \mathfrak{e}_8)$.

\subsubsection{Algebraic enumerative geometry}
Let $\cO_{\bbP^1}(-1)$ denote the tautological bundle on the complex
projective line, and let $X_{\mathfrak{g}} \triangleq [\cO(-1)^{\oplus 2}_{\bbP^1}/\Gamma_{\mathfrak{g}}]$
be the fibrewise quotient stack of $\mathrm{Tot}(\cO(-1) \oplus \cO(-1) \to \bbP^1)$ by the
action of the finite group $\Gamma_{\mathfrak{g}}$ as in \cref{sec:cs}. The {\it Gromov--Witten
potential} of $X_{\mathfrak{g}}$ is a formal generating series of virtual counts of
twisted stable curves to the orbifold $X_{\mathfrak{g}}$, to all genera, degrees, and
insertions of twisted cohomology classes \cite{MR2450211}. It was proposed and
non-trivially checked in \cite{Borot:2015fxa, Brini:2017gfi}
that the full Gromov--Witten potential of $X_{\mathfrak{g}}$ coincides with the
Chern--Simons partition function of $M_{\mathfrak{g}}$; therefore the full curve
counting information on $X_{\mathfrak{g}}$ is encoded into \eqref{eq:diagsc}, and
solved by \cref{prob:gen}, as an instance of 1-dimensional mirror symmetry for
Calabi--Yau manifolds.

\subsubsection{Orbits of affine Weyl groups and Frobenius manifolds}
\label{sec:polfrob}
A suitable degeneration\footnote{This is essentially given by a suitably defined $\lambda \to
  \infty$ limit of \eqref{eq:uilambda}.} of the family \eqref{eq:diagsc} was proved in
\cite{Brini:2017gfi} to provide a new 1-dimensional mirror for the orbifold quantum co-homology of the
orbifold projective line $\bbP^1_{\mathfrak{g}} = [\bbC^* \GITl{} \bbC^2 \GIT{} \Gamma_{\mathfrak{g}}
]$, or, equivalently \cite{MR2672302}, to the Frobenius manifold structure on the orbits of the
affine Weyl group of the same ADE type of $\Gamma_{\mathfrak{g}}$ \cite{MR1606165}. 
 This allowed \cite{Brini:2017gfi} to solve the long-standing problem
 \cite{MR1606165} of determining flat
co-ordinates for the Saito metric for all ADE types, and gave a higher genus reconstruction
theorem by the Chekhov--Eynard--Orantin topological recursion as a bonus. Once
more, the central tool in the theorem is the closed-form presentation of the
family of spectral curves \eqref{eq:diagsc}, and therefore, the solution of
\cref{prob:gen} in all ADE types; particularly, $(\cG,V)=(E_8, \mathfrak{e}_8)$.

\subsubsection{$\bbQ$-extensions with Galois group $\cW(E_8)$}
\label{sec:galois}
A subject of particular interest in arithmetic geometry is
the construction of extensions of the rationals with Galois group equal to the
full Weyl group of $E_8$ \cite{MR1123543,MR2559115}. By a theorem of
Shioda \cite[Thm~.7.2]{MR1123543}, the Galois action on the extension arises
from the action of the full Weyl group of $E_8$ on the Mordell--Weil lattice of a non-isotrivial elliptic
curve $E$ over $\bbQ(t)$, and in turn, from the datum of a degree 240, monic
integer polynomial $\Psi \in \bbZ[x]$ whose splitting field has
$\cW(E_8)$ as Galois group: a few explicit instances of these
polynomials were constructed in \cite{MR3117502, MR2561902, MR2523316}. Solving
\cref{prob:gen} for $(\cG,V)=(E_8, \mathfrak{e}_8)$ gives manifestly a
$\bbZ^8$-worth of candidate such integer polynomials\footnote{The path followed
in \cite{MR2523316} to find one such example hinges on constructing such a
polynomial precisely as an adjoint characteristic polynomial for $E_8^\bbC$, upon factoring out the
contribution of zero roots.} upon specialising $\chi_{i} \in \bbZ$: and
for generic integral values \cite{MR3117502}, the resulting polynomial has the full Weyl group of $E_8$
as Galois group\footnote{For sufficient conditions
  for this to happen, see e.g. \cite[Lemma~3.2]{MR2523316}.}. The solution of \cref{prob:gen}
generates in this way an infinite wealth of examples of multiplicatively excellent
families of elliptic surfaces (in the sense of \cite{MR3117502}) of type $E_8$.

\subsection{Sketch of the solution}

I briefly describe here the strategy employed to solve \cref{prob:gen} for
$(\cG,V)=(E_8, \mathfrak{e}_8)$, which was announced in my previous work \cite[Appendix~C]{Brini:2017gfi}.
In its crudest terms, what we would like to achieve amounts to enforcing
\beq
\chi_{\wedge^k \mathfrak{e}_8} = \mathfrak{p}_k(\chi_{1}, \dots, \chi_{8}) \in \bbZ[\chi_{1},
  \dots, \chi_{8}]
\label{eq:chikpk}
\eeq
for an unknown polynomial $\mathfrak{p}_k$ as an identity of regular functions on the Cartan torus
-- that is, as an identity between integral Laurent
polynomials. However, a direct calculation of $\mathfrak{p}_k$ in this vein is unviable
because of the exceptional complexity of $E_8$, even for small values of
$k$. We
find a
workaround that breaks up the task of determining $\mathfrak{p}_k$ in
\eqref{eq:chikpk} into three main
steps:
\begin{description}
\item[Casimir bound and finite-dimensional reduction]
  An {\it a priori} bound on the number of
  monomials appearing in $\mathfrak{p}_k$ in \eqref{eq:chikpk} holds. The set of allowed monomials
has cardinality $\approx 10^6$, and \eqref{eq:chikpk} reduces then to a finite-dimensional linear problem of
  the same rank upon taking sufficiently many numerical
  ``sampling points'' $g \in \cT$ and evaluating $\chi_{\wedge^k \mathfrak{e}_8}(g)$ and
  $\chi_{i}(g)$ at $g$. 
%
%
\item[Partition of the monomial set]
  For generic sampling sets,
  the resulting linear problem is size-wise three orders of
magnitude beyond the reach of practical calculations.  There exist however special choices of the sampling
  set such that the original linear problem is equivalent to $\approx 4\times
  10^3$ linear sub-systems of size varying from one to $\approx 3\times  10^3$. This can be realised by constructing the sampling set numerically
  using Newton--Raphson inversion, evaluating derivatives of characters in
  exponentiated linear co-ordinates on the torus up to sufficiently high order, and
  establishing rigorous analytic bounds to perform an exact integer rounding.

\item[Partition of the sampling set and distributed computing] The problem can
  then be solved effectively on a computer: most of the
  runtime comes from the construction of the linear subsystems, by carrying
  out the generation of the sampling set and
  the evaluation of derivatives of $\chi_{i}$ and $\chi_{\wedge^k
    \mathfrak{e}_8}$, and it is in the order of a few years. On the other hand the  calculation can be effectively
  parallelised by a suitable segmentation of the sampling set, and subsequent distribution among
 different processor cores:
%
 this allowed to reduce in our computer implementation the total absolute clock-time 
taken by the entire computation to about six weeks on a small departmental cluster. The
  result is available as a computer package at
  \begin{center}
    \large
  \texttt{http://tiny.cc/E8Char} .
  \end{center}
  and a description of the individual files is given in \cref{sec:auxfiles}.
\end{description}

The last two points indicate quite clearly that this project had a very substantial computational
component to it. I discuss rather diffusely the details of its concrete implementation in
\cref{sec:compute,sec:computeapp}.

%
%

\subsection{Organisation of the paper}

The paper is organised as follows. In \cref{sec:extchar} we review the problem
and explain the Casimir bound that reduces \cref{prob:gen} to a
finite-dimensional, generically dense linear problem of rank $\approx 10^6$; the
Section ends with the central statement (\cref{thm:mainthm}) that suitable choices of
sampling sets lead to a block-reduction of this linear
problem to $\approx 4\times 10^3$ linear sub-systems 
over the
rationals by suitable partitions of the
set of admissible monomials; both the reduction and
the solution of these linear problems can be performed effectively by parallel
computation. This statement is justified in \cref{sec:computeapp}, which
occupies the main body of the paper: after reviewing general computational
strategies in \cref{sec:4paths}, I describe the
partition of the original problem in \cref{sec:partition}, and the exact computation
of the reduced linear problems
using semi-numerical methods and analytic bounds
for exact integer roundings in \cref{sec:complphi,sec:compbphi,sec:solvext}; additional details of
the computer implementation are given in \cref{sec:impl}. Finally, \cref{sec:appl}
contains some applications to the construction of explicit integral
polynomials with Galois group $\mathrm{Weyl}(E_8)$; the reader is referred to
\cite{Brini:2017gfi} for further applications of the results obtained here to
integrable systems ($E_8$ relativistic Toda lattices), $E_8$ Seiberg-Witten
theory, quantum invariants of the Poincar\'e sphere, and Frobenius manifolds
on orbits of the affine Weyl group of type $E_8$. Use and abuse of notation throughout the text will be in accordance with \cref{tab:notation}.

\begin{table}[!h]
\begin{tabular}{|c|c|}
\hline
$E_8$ & \makecell{The Dynkin type specified by the diagram of
  \cref{fig:dynke8} \\ The complex Lie group of the same Dynkin type}  \\
\hline
$\mathfrak{e}_8$ & \makecell{The Lie algebra of type $E_8$ \\The adjoint
  representation $\mathrm{Ad}_{E_8}$ of $E_8$} \\
\hline
$\mathfrak{h}_{\mathfrak{e}_8}$ & The Cartan subalgebra of $\mathfrak{e}_8$ \\
\hline
$\cT_{E_8}$ & The Cartan torus $\cT_{E_8}=\exp(\mathfrak{h}_{\mathfrak{e}_8})$ of $E_8$ \\
\hline
$g$ & A regular element of $E_8$ \\
\hline
$\cW(E_8)$ & The Weyl group $N_{\cT_{E_8}}/\cT_{E_8}$ \\
\hline
$h$ & A regular element of $\mathfrak{h}_{\mathfrak{e}_8}$\\
(resp. $\exp(h)$) & (resp. an element of a Weyl orbit in $\cT_{E_8}$ conjugate
to a regular $g\in E_8$) \\
\hline
$\Pi=\{\a_1, \dots, \a_8\}$ & The set of simple roots of $\mathfrak{e}_8$\\
\hline
$\Omega=\{\omega_1, \dots, \omega_8\}$ & The set of fundamental weights of $\mathfrak{e}_8$  \\
\hline
$\Delta$ (resp. $(\Delta^\pm)$) & The set of all (resp. positive/negative) roots of $\mathfrak{e}_8$\\
\hline
$V_\omega$ & The irreducible representation with highest weight $\omega$ \\
\hline
$V_i$ & The $i^{\rm th}$ fundamental representation, $i=1, \dots, 8$ \\
\hline
$\Gamma(V)$  & The weight system  of the representation $V$ \\
(resp. $\Gamma_{\rm red}(V)$) & (resp. the
weight system without multiplicities)\\
\hline
$\chi_\omega$ & The formal character of  $V_\omega$  \\
\hline
$\chi_i$ & The formal character of $V_{i}$  \\
\hline
$\rho$ & The Weyl vector of $\mathfrak{e}_8$  \\
\hline
$(t_1, \dots, t_8)$ & Linear co-ordinates on the Cartan subalgebra of $\mathfrak{e}_8$ w.r.t. the
co-root basis $\Pi^*$ \\
\hline
$(Q_1, \dots, Q_8)$ & $(\exp (t_1), \dots, \exp (t_8))$\\
\hline
$D^c$ & $\frac{\de^{|c|}}{\de Q_1^{c_1} \dots \de Q_8^{c_8}}$ \\
\hline
$\mathfrak{D}^c$ & $\frac{\de^{|c|}}{\de \chi_1^{c_1} \dots \de \chi_8^{c_8}}$ \\
\hline
$|v\in \bbZ^8|$ & $\sum_{i=1}^8 |v_i|$ \\
\hline
$C_{E_8}$ & The Cartan matrix of $E_8$ \\
\hline
\end{tabular}
\caption{Notation employed throughout the text.}
\label{tab:notation}
\end{table}

\subsection*{Acknowledgements}
I am grateful to G.~Borot and A.~Klemm for useful discussions on related topics, and to
A.~D'Andrea, C.~Bonnaf\'e, M.~Coti~Zelati, D.~Craven, S.~Goodwin and T.~Weigel for taking the time to answer my
questions at various stages of progress of this work. I am indebted to
F.~David--Collin and especially
B.~Chapuisat and A.~Thomas for patiently dealing with my incessant requests
for libraries and resources to be installed, respectively, on the
\texttt{Omega} departmental cluster at the Institut Montpelli\'erain Alexander
Grothendieck in Montpellier, and the Maths Compute Cluster at Imperial College
London, where most of the debugging, tests, and parallel calculations leading up to the results of
this work were carried out. Partial computing support
from the interfaculty HPC@LR computing cluster \texttt{Muse} at the University
of Montpellier is also acknowledged. I made extensive
use of the GNU~C~libraries GMP, MPFR, and MPC \cite{gmp, mpfr, mpc} for the
use of arbitrary precision floating-point arithmetics over the complex numbers, and of
the FLINT~C~library \cite{flint} for multi-precision arithmetics over the
integers and the rationals; C~source code for all calculations is available upon request. This research
was partially supported by the ERC Grant no. 682603 (PI:
T. Coates) and the EPSRC Fellowship grant EP/S003657/1.

\begin{figure}[t]
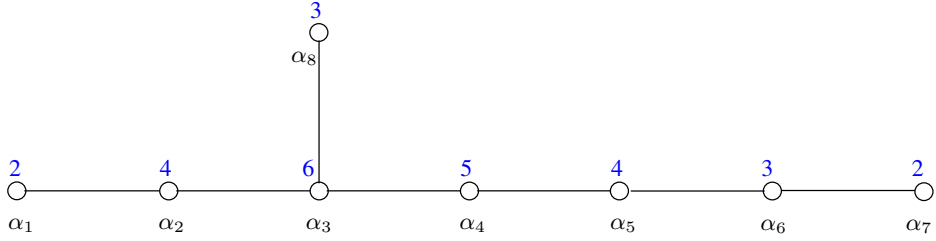
\caption{The Dynkin diagram of type $\mathrm{E}_8$. The numbers in blue are
  the components of the highest root in the $\alpha$-basis.}
\label{fig:dynke8}
\end{figure}

\section{Exterior character relations}
\label{sec:extchar}

\subsection{Introduction}

We start by recasting \cref{prob:gen} in the following form. The cofactor expansion of the characteristic polynomial reduces
the problem to finding polynomials $\mathfrak{p}_k \in \bbZ[\chi_{1}, \dots, \chi_{8}]$
such that 
\beq
\chi_{\wedge^k \mathfrak{e}_8} = \mathfrak{p}_k(\chi_{1}, \dots, \chi_{8}), 
\label{eq:chikpk2}
\eeq
as in  \eqref{eq:chikpk}.
\begin{defn} Let $V \in \mathrm{Rep}(\cG)$. We call the presentation of $\chi_V$ as a polynomial in the
  generators $\chi_{k}$ of the representation ring the {\rm polynomial
    character decomposition} for $\chi_V$.
\end{defn}

Computing the polynomial character decomposition of $\chi_{\wedge^k
  \mathfrak{e}_8}$ can be simplified somewhat for large $k$. Since the adjoint representation is real, we have
$\chi_{\wedge^k \mathfrak{e}_8}=\chi_{\wedge^{248-k} \mathfrak{e}_8}$, meaning we
can restrict ourselves to $k\leq 124$. Moreover, the characteristic polynomial in any
representation factors as a product over irreducible factors corresponding to
Weyl orbits: in the case of the adjoint, we have
\beq
\det_{\mathfrak{e}_8}(g-\mu \mathbf{1})= (\mu-1)^8 \mathfrak{Q}_{240}(\mu; g)
\label{eq:q240}
\eeq
for a degree 240 polynomial  $\mathfrak{Q}_{240}(\mu;
g)$, corresponding to the contribution of non-zero roots. Then \eqref{eq:q240}
imposes polynomial (in fact linear) relations in $\mathrm{Rep}(E_8)$,
\beq
\chi_{\wedge^j \mathfrak{e}_8} \in \bbQ[\chi_{\mathfrak{e}_8, \wedge^2
    \mathfrak{e}_8, \dots, \wedge^{120} \mathfrak{e}_8}], \quad j=121, \dots,
124,
\label{eq:124to120}
\eeq
meaning that in \eqref{eq:chikpk2} we only\footnote{It is natural to ask whether we could strengthen \eqref{eq:124to120} for
$k<120$ (and possibly all the way
down to $k=9$), i.e. whether
\beq
\chi_{\wedge^j \mathfrak{e}_8} \in \bbQ[\chi_{\mathfrak{e}_8, \wedge^2
    \mathfrak{e}_8, \dots, \wedge^{8} \mathfrak{e}_8}], \quad j=9, \dots,
124,
\eeq
I am unable to conclude from the structure of the root system of $E_8$, nor do I in fact expect, that such a statement holds.} need to worry about $k\leq 120$. 
We can
then equivalently phrase \cref{prob:gen} for $(E_8, \mathfrak{e}_8)$ as follows.

\begin{prob}
For all $1\leq k \leq 120$ and $\iota \in \bbN_0^8$, determine $N^{(k)}_{\iota} \in \bbZ$ such
that
\beq
\chi_{\wedge^k \mathfrak{e}_8} = \sum_{\iota} N^{(k)}_{\iota} \prod_{l=1}^8 \chi_{l}^{\iota_l}.
\label{eq:e8num}
\eeq
\label{prob:e8num}
\end{prob}
There is in principle a systematic recursive method to determine
$N^{(k)}_{\iota}$, as follows. Note that $\wedge^k \mathfrak{e}_8$
is highly reducible as soon as $k>1$; let
\beq
\wedge^k \mathfrak{e}_8 = \bigoplus_{i=1}^{f_k} \mathtt{m}^{(k)}_i
V_{\sum_l j^{(k)}_{i,l} \omega_l}
\label{eq:redgen}
\eeq
be its decomposition into irreducible summands, with $\mathtt{m}^{(k)}_i, f_k,
j^{(k)}_{i,l} \in \bbN$.
%
Then \cref{prob:e8num} can be solved in two steps: by determining
first the multiplicities $\mathtt{m}^{(k)}_i$, and then solving
\cref{prob:e8num} for each irreducible summand in \eqref{eq:redgen}. This last step can be
achieved recursively by computing, for a dominant weight $\varpi=\sum_k \varpi_k
\omega_{k}$, the product
\beq
V_{\varpi-\omega_{\bar k}} \otimes
V_{\omega_{\bar k}} = V_{\varpi} \oplus  \bigoplus_{b=1}^{f_{\varpi}}
\mathtt{m}^{\varpi}_i V_{\sum_l j^{\varpi}_{i,l}\omega_l}
\label{eq:lrdec}
\eeq
where $\mathtt{m}^{(\varpi)}_i, f_\varpi,
j^{\varpi}_{i,l} \in \bbN$ with notation adapted from \eqref{eq:redgen}, and
$\bar k$ is an index labelling the non-vanishing component of $\varpi$ in the direction of a
fundamental weight with weight system of minimal cardinality:
\beq
\bar k \in I(\varpi) := \{k=1, \dots, 8 | \varpi_k \neq 0\} \quad \mathrm{s.t.} \quad  \dim{V_{\omega_{\bar  k}}} = \mathrm{min}_{j \in I(\varpi)} \dim{V_{\omega_j}}.
\eeq
%
%
This is easily shown to be a recursion
with induction data given by the fundamental weights and the zero weight.
\begin{example}
Take $k=3$. We first decompose $\wedge^3 \mathfrak{e}_8$ into irreducibles to get
\beq
\wedge^3 \mathfrak{e}_8= V_{0} +
V_{\omega_5} + V_{\omega_6} + V_{2\omega_7}
+ V_{\omega_1}
\eeq
Upon taking characters, we get
\beq
\chi_{\wedge^3 \mathfrak{e}_8}= 1 +
\chi_{5} + \chi_{6} + \chi_{2\omega_7}
+ \chi_{1}
\label{eq:decred3}
\eeq
so the only bit we need to compute is the polynomial decomposition of $\chi_{2
  \omega_7}$ in the tensor algebra. In this case we have only one
non-vanishing component of $\varpi :=2 \omega_7$, and we can compute
explicitly from \eqref{eq:lrdec} that
\beq
V_{\varpi-\omega_7} \otimes V_{\omega_7} =
V_{\omega_7} \otimes V_{\omega_7} =
 V_{2\omega_7} \oplus V_{0} \oplus V_{\omega_7} \oplus  V_{\omega_1} \oplus V_{\omega_6}.
\label{eq:lr3}
\eeq
Combining \eqref{eq:decred3} and \eqref{eq:lr3} and taking characters gives
\beq
\chi_{\wedge^3 \mathfrak{e}_8}= \chi_{5} -\chi_{7} +\chi_{7}^2.
\eeq
\label{ex:dec}
\end{example}
%
While this example is relatively simple, carrying out the same procedure for
$k$ all the way up to $k=120$ is unfeasible, both for the growth in complexity of the tensor product decompositions in
\eqref{eq:lrdec}  and the recursive procedure outlined above as $k$ grows, 
and especially because of the difficulty in computing explicitly the multiplicities of plethysms
such as \eqref{eq:decred3} for all $k$, which is a hard problem in its own right (and
one that is 
unsolved to date, see \cite{MR1822681,MR0167567,MR2069810,MR1437204}
for an account).

So we need to circumvent both problems, and possibly, we need to do so in a way
that works uniformly with respect to the exterior exponent $k$. I will
describe in the next section four alternative methods to accomplish this, each having its upsides and
shortcomings.

\subsection{Character decompositions and an a priori bound}

  I will start by establishing sufficient vanishing conditions
for the numbers $N^{(k)}_{\iota}$ in \cref{prob:e8num}. Let
\beq
e_{jk} := \sum_j (C_{E_8}^{-1})_{jk}
\eeq
be the $k^{\rm th}$ component of $\omega_j$ in the root basis. Notice that
this is always a positive integer, since $E_8$ has trivial centre.
\begin{defn}
  We call the finite set
  \beq
  \mathfrak{I}:= \l\{\iota \in (\bbZ_+)^8 \bigg| \sum_j \iota_j
  e_{jk} \leq 2 \sum_j e_{jk} \r\}
  \label{eq:admexp}
  \eeq
the {\rm set of admissible exponents} of the exterior algebra $\wedge \mathfrak{e}_8$.
\end{defn}
%
%
\begin{lem}
 We have that 
 \beq
 \iota \notin \mathfrak{I} \Rightarrow 
N^{(k)}_{\iota}=0 \quad \forall k.
\label{eq:condmon}
\eeq
\label{lem:vanish}
\end{lem}
\begin{proof}
 By definition of the exterior power, weights in $\Gamma(\wedge^k \mathfrak{e}_8)$ have the form $\mu_{i_1 \dots i_k}=\sum_{j=1}^k
 \alpha_{i_j}$ with $i_j \neq i_l$ for $j\neq l$. When $k=120$, there is a
unique element in the weight module with $i_j>0$ $\forall j=1, \dots, 120$,
which is precisely $2\rho=\sum_{\a \in \Delta^+} \alpha$. Therefore, for any
$\mu \in \Gamma(\wedge^{120}\mathfrak{e}_8)$, we have that $2\rho-\mu
= \sum_{i=1}^8 n_i \alpha_i$ with all $n_i \in \bbZ_{\geq 0}$: equivalently,
for any $\mu$ in $\Gamma(\wedge^k \mathfrak{e}_8)$, we must have
$\mu \preceq 2 \rho$, where $\lambda \preceq \mu$ is the usual partial order
on the weights indicating that  $\mu-\lambda$ is a non-negative integral linear
combination of the simple roots. Consider now the representation space version of \eqref{eq:e8num},
\beq
\wedge^k \mathfrak{e}_8 = \bigoplus_{\iota} N^{(k)}_{\iota} \bigotimes_{l=1}^8 V_{\omega_l}^{\iota_l},
\label{eq:e8numrep}
\eeq
and consider the subset of indices $\iota_j\in \bbN_0$ such that 
$\sum_{j}\iota_j \omega_j \npreceq 2 \rho$. Since the corresponding
one-dimensional weight space must appear with zero coefficient on the
l.h.s. of \eqref{eq:e8numrep}, the weighted occurences of this weight space 
must sum up with total coefficient equal to zero on its r.h.s..  Consider now
the $\bbZ$-linear map 
\bea
\mathfrak{f} : \bbN^8 & \longrightarrow & \bbN^{10} \nn \\
\{\iota_j\}_{j=1}^8 & \longrightarrow & \{k_l\}_{l=1}^{10}
\eea
induced by the factorisation
\beq
\prod_j\mathrm{dim} V_{\omega_j}^{\iota_j}= 2^{k_1} 3^{k_2} 5^{k_3} 7^{k_4} 11^{k_5} 13^{k_6}
17^{k_7} 19^{k_8} 23^{k_9} 31^{k_{10}}.
\eeq
It is immediate to verify that
$\mathfrak{f}$ maps the positive 8-orthant injectively into the 10-orthant, as
the reader can verify by writing down explicitly the expression of
$\mathfrak{f}_l(\iota)=\sum_{lj}\mathsf{m}^\mathfrak{f}_{lj} \iota_j$ and
verifying that the integral $10\times 8$ matrix $\mathsf{m}^\mathfrak{f}$ has
maximal rank. This implies that each summand $\bigotimes_l
V_{\omega_l}^{\iota_l}$ in \eqref{eq:e8numrep} is
uniquely determined by the exponents $\iota_l$ of the tensor powers. Let then $\iota_j^{\rm max}$ be the
indices labelling the summand  of maximal
dimension on the r.h.s. of \eqref{eq:e8numrep}.  This summand contains the dominant weight
$\sum_{j}i^{\rm max}_j \omega_j\npreceq 2 \rho$ in its weight module, and by
construction it is the unique direct summand in \eqref{eq:e8numrep}
containing it, so $N^{(k)}_{\iota}=0$. Induction on the maximal dimension gives
$N^{(k)}_{\iota}=0$ for all $\sum_{j}\iota_j \omega_j\npreceq 2
\rho$. The lemma then follows from writing down the condition $\sum_{j}\iota_j \omega_j\preceq 2
\rho$ in the root basis.
\end{proof}

\begin{rmk}
  It is obvious that \eqref{eq:condmon} could be further refined to provide
  stronger bounds for any fixed $k<120$. It will suffice in our discussion to
  consider just the uniform bound above, and treat all values of $k$
  simultaneously in \cref{prob:e8num}, as will be clear momentarily.
\end{rmk}

\begin{rmk} The condition of being admissible gives an {\it a priori} bound on the range of $\iota_j$ in the sum
on the right hand side of \eqref{eq:e8num},
\beq
\chi_{\wedge^k \mathfrak{e}_8} = \sum_{\iota \in \mathfrak{I}} N^{(k)}_{\iota} \prod_{l=1}^8 \chi_{l}^{\iota_l}
\label{eq:e8num2}
\eeq
with possibly at most $|\mathfrak{I}|=950077$ monomials appearing with non-zero
coefficient in \eqref{eq:e8num2} for any given value of $k$. Crude as it might
appear, the sufficient vanishing condition of \cref{lem:vanish} is remarkably
close to be necessary as well for sufficiently big $k$; we will be able to
check {\it a posteriori} that 
\begin{align}
  \mathrm{card}\Big\{ \iota \in \mathfrak{I} \big| N^{(k)}_{\iota} \neq 0 \hbox{ for some } k  \Big\} & = 949468,\\
  \mathrm{sup}_k \mathrm{card}\Big\{ \iota \in \mathfrak{I} \big| N^{(k)}_{\iota} \neq 0  \Big\} & = 949256 \quad (k=118),
\end{align}
meaning that most ($\approx 99.94 \%$) admissible monomials will appear with
non-vanishing coefficient in the
polynomial character
decomposition of $\chi_{\wedge^k \mathfrak{e}_8}$ for some $k$, and most
(up to $\approx 99.91$ \%) will appear for a fixed $k$ big enough.

\end{rmk}

\subsection{Computing the solution of \cref{prob:e8num}}
\label{sec:compute}

\cref{lem:vanish} reduces \cref{prob:e8num} to a finite dimensional linear
  problem which can in principle be fed to a computer solver; in practice, however,
  this will require a substantial degree of additional sophistication.
  The main 
  idea of our solution is
  that  the large linear problem \cref{prob:e8num} is non-trivially equivalent
    to a big number of small linear problems over $\bbQ$, that are explicitly solvable in
    parallel (and within our lifetimes; typically in the order of a few weeks
    when implemented on a small compute cluster).
%

  \subsubsection{Four paths to polynomial character decompositions}
  \label{sec:4paths}

  \cref{lem:vanish} reduces \cref{prob:e8num} to the following finite dimensional linear
problem: impose \eqref{eq:e8num2} as an identity of regular
functions on $\cT_{E_8}$,
\beq
\chi_{V} = \sum_{w \in \Gamma(V)} \prod_i Q_i^{w_i} \in \bbZ[Q_1^\pm, \dots, Q_8^\pm],
\label{eq:Qpol}
\eeq
and then read off $N^{(k)}_{\iota}$ by plugging $V=\wedge^k
\mathfrak{e}_8$ and $V=V_{_j}$ onto either side of \eqref{eq:e8num2}, and 
equating the coefficients. One immediate drawback of such a brute-force
approach is the sheer size of these Laurent
polynomials -- from \cref{lem:vanish}, there are $\prod_k (1+2 \sum_j
e_{jk})\approx 3.2 \times 10^{17}$ weight spaces appearing in \eqref{eq:e8num2}, thus
rendering this approach unwieldy. I describe below four computational methods that are both
more effective and are also complementary in what they allow to compute.

\begin{description}
\item[Method 1 (Numerical sampling)] The simplest thing to do is to consider a {\it sampling set} of $|\mathfrak{I}|$ numerical
  values for the conjugacy class $[g]=[\exp(h)]=[(Q^{(\iota)}_1, \dots, Q^{(\iota)}_8)] \in
  \cT_{E_8}/\cW_{E_8}$ for $i\in \mathfrak{I}$. Upon evaluating the regular
  characters $\chi_{\wedge^k \mathfrak{e}_8}(\exp h)$, $\chi_{j}(\exp
  h)$ on either side of \eqref{eq:e8num2}, and for generic sampling sets, we
  are  left with a rank-$|\mathfrak{I}|$ linear system with $N_\iota^{(k)}$ as
  unknowns to solve for.  \\
  This is a straightforward method which furthermore works uniformly in
  $k$. However, for a generic choice of sampling values, the linear system
is dense and extremely poorly-conditioned -- the corresponding matrix is a
generalised Vandermonde matrix minor -- which implies that exact arithmetics
is required for its solution. Memory-wise this places a bound for the size at around
\beq
r_{\rm max} := 3.5\times 10^3.
\eeq
Since $|\mathfrak{I}|\approx 10^6$, this means that we're off by three orders of magnitude here.
  
\item[Method 2 ($Q$-expansions)] As an alternative, we could consider to
  Taylor-expand $\chi_{\wedge^k \mathfrak{e}_8}(g)$ in $\chi_i$ at a value $g=\exp h$ such that
  $\chi_{i}(g)=0$ for all $i$. It is immediate to verify, for
  example, that
\beq
t(h_0)=-\frac{2 \ri \pi }{31}\big(6,3 ,15,1,12,-4,5,0\big)
\label{eq:u0}
\eeq
in linear co-ordinates for $\cT_{E_8}$ results in $\chi_i(\exp h_0)$
being zero for all $i$:
a direct calculation of $\chi_{\wedge^k \mathfrak{e}_8}(\exp h_0)$
in \eqref{eq:e8num2} returns the constant term $N^{(k)}_{(0,\dots, 0)}$ for all
$k$. In the same vein, $N^{(k)}_{\iota}$ can be computed by taking
higher order $\chi_{i}$-derivatives of $\chi_{\wedge^k
  \mathfrak{e}_8}$ at $h=h_0$ in the following three steps:
\ben
\item we first compute $Q_j$-derivatives of $\chi_{\wedge^k
  \mathfrak{e}_8}$ at $h=h_0$ from \eqref{eq:Qpol}, using the explicit
  structure of $\Gamma(\mathfrak{e}_8)=\Delta$;
\item the Jacobian $\de_{Q_j} \chi_{i}$ may be computed combining basic relations in $\mathrm{Rep}(E_8)$ with Newton
  identities;
\item finally, higher order $\chi_{j}$ derivatives of $\chi_{\wedge^k
    \mathfrak{e}_8}$ can be obtained by the multi-variate Fa\`a di Bruno
  formula\footnote{ We discuss a version of this formula in \cref{prop:FdB}.} \cite{MR1325915}, converting $\de^{|\iota|}_{\chi_{1}^{\iota_1}\dots \chi_{1}^{\iota_8}}$ into
    its expression as a $|\iota|^{\rm th}$-order differential operator in
    $(Q_j)_j$.
\een
    The main drawback
    here is 
    the factorial growth in $|\iota|$  of the number of terms in the Fa\`a di Bruno
    formula, which practically limits its use to
    \beq
    |\iota|\lesssim 5 =: d_{\rm max}.
    \eeq
This is
a very long way from the value
\beq
|\iota^{\rm max}|:=\sup_{\iota\in \mathfrak{I}}
\sum_j \iota_j=31
\label{eq:iotamax}
\eeq
for the highest sum of admissible exponents in \eqref{eq:admexp}.
\item[Method 3 (Numerical interpolation)] A further, numerical way to determine $N^{(k)}_{\iota}$ is as
  follows. Suppose that, for every $\iota$, we could find a solution $h_\iota$ of the system of algebraic
  equations in $Q$
  \beq
  \chi_{{j}}(\exp h_\iota) = \sum_{w \in \Gamma(V_{\omega_j})} \prod_l Q_l^{w_l}=u_{j}^{(\iota)}
  \label{eq:invchi}
  \eeq
  for given $\{(u_1, \dots, u_8)^{(\iota)} \in \bbC^8 \}_{\iota \in \mathfrak{I}}$. Then $\mathfrak{p}_k(\chi_{1}, \dots, \chi_{8})$ can be recovered by
  polynomial interpolation of $\chi_{\wedge^k \mathfrak{e}_8}(Q(u^{(\iota)}))$
  in $\chi_{1}, \dots, \chi_{8}$ through the given values of $\chi_j(\exp h)=u_j^{(\iota)}$:
  even if the solution $Q(u_{i}^{(\iota)})$  cannot in general be found
  exactly, one might still hope to compute an approximate, generic root of the algebraic
  system  \eqref{eq:invchi} numerically; and if this is done with sufficiently
  high precision, the
sought-for interpolating coefficients $N^{(k)}_{\iota}$ can be
reliably computed by rounding to the nearest integer. \\
This semi-numerical approach is effective in computing e.g. $N^{(k)}_{\iota}$
when all but one $\iota_j \neq 0$: in this case interpolation is done in only one variable
and the number of sampling points necessary is in the order of a few dozens. The main
problem with the general case is that it requires $|\mathfrak{I}|$
interpolating points, hence we would be faced with a large interpolating
matrix to invert and the same 
problem of {\bf Method 1} above. Alternatively, we could resort to interpolation on the
lattice hypercube $\times_j [0, \iota_j^{\rm max}] \cap \bbZ^8$ with $\iota_j^{\rm
  max}=\mathrm{sup}_{\iota\in\mathfrak{I}}\iota_j$, with a trivially invertible
Vandermonde matrix but $\prod \iota_j^{\rm max}\approx 6.5 \times 10^9$ high-precision
numerical inversions to perform. On the type of computer systems we
employed\footnote{These were all 64-bit systems with quad-core
  CPUs and up to 48 Gb of
  RAM.}, a single numerical root is typically found in
$\approx 10~\mathrm{sec}$ to the necessary precision, so
the total runtime would take $\approx 2\times10^3~\mathrm{years}$ on a single CPU.

\item[Method 4 ($p$-adic expansions)] We could also consider a version of {\bf
  Method 1} where we sample at
  integers and reduce mod $p$ for sufficiently large $p$. This works efficiently
  for low $k$, but it is a challenge to determine the behaviour for higher
  exterior powers and, furthermore, it is hard to place an {\it a priori} bound on the largest $p$ required.
\end{description}

Each of these methods individually fails to address \cref{prob:e8num}; we have
in particular bounds
\beq
r_{\rm max} \approx 3.5\times 10^3, \qquad d_{\rm max} \approx 5
\label{eq:bounds}
\eeq
for the size of the linear systems we can solve and the order of the
derivatives respectively for {\bf Method 1} and {\bf Method 2}: the bound on the size
comes from memory constraints when using exact linear solving algorithms, whilst
the bound on the order is a runtime bound for the use of the Fa\`a di Bruno
formula in many variables.

Nonetheless, it is still possible to
combine {\bf Methods 1-4} in such a way that a complete answer of the problem
can be obtained efficiently by distributed computation. Suppose that $(Q^{(\kappa)}_1,
\dots, Q^{(\kappa)}_8) \in
\cT_{E_8}$ is a point on the Cartan torus such that the value of the $j^{\rm th}$
fundamental character, $u_{j}^{(\kappa)}:=\chi_{j}(Q^{(\kappa)})$ in \eqref{eq:invchi}, is zero for all but a few values of
$j$. Then, upon evaluating $\chi_{j}$ at such $Q^{(\kappa)}$, all terms
with $\iota_j\neq 0$ drop out of the r.h.s. of \eqref{eq:e8num2}, leaving a
 a smaller number of unknown coefficients $N_\iota^{(k)}$ to solve for; case
 in point, \eqref{eq:u0} is an example where only one term survives.

This means
that if we could choose judiciously the elements $(Q_1^{(\kappa)}, \dots, Q_8^{(\kappa)})$ in {\bf Method~1} such
that any given subset of the fundamental characters are vanishing at $Q^{(\kappa)}$, the original linear
problem \eqref{eq:e8num2} can be reduced to $2^8$~linear sub-systems of
smaller size: any $\phi\in\bbN^8$ with $\phi_j\in \{0,1\}$ determines a subset $\mathfrak{I}_\phi\subset\mathfrak{I}$ with $\phi_j$ equal to zero or one for $\iota_j=0$ or
$\iota_j>0$ respectively, and this is a partition of the set of admissible
exponents. We can further reduce the size of these systems by using {\bf
  Method~2} and considering derivatives
$\frac{\de^{|c|} \chi_{\wedge^k \mathfrak{e}_8}}{\de \chi_{1}^{c_1}\dots \de
    \chi_{8}^{c_8}}(Q_1^{(\kappa)}, \dots, Q_8^{(\kappa)})$: indeed, computing
  $\chi_{j}-$derivatives of given order $c_j$ of the l.h.s. of \eqref{eq:e8num2} and evaluating at points
  with $u_j^{(\kappa)}=\chi_{j}(Q^{(\kappa)})=0$ entails that only summands with
  $\iota_j=c_j$ will appear on
  the r.h.s.. I claim that the procedure above
  can be performed effectively and in a way
  compatible with \eqref{eq:bounds} in the following theorem, the proof of
  which will occupy the largest part of \cref{sec:computeapp}.

  As above, let $\phi\in\{0,1\}^8$ denote
  an 8-tuple of integers equal to either $0$ or $1$, and let $i_1(\phi)$ (resp. $i_2(\phi)$) be the
  index of the non-vanishing component of $\phi$ such that $\dim
  V_{\omega_{i_1(\phi)}}$ is maximal (resp. $\dim
  V_{\omega_{i_2(\phi)}}$ is next-to-maximal) in $\big\{\dim
  V_{\omega_i}\big\}_{i | \phi_i \neq 0}$; e.g.
  \beq
  \phi=(0,1,0,0,1,0,1,1) \Rightarrow i_1(\phi)=2;~ i_2(\phi)= 5.
  \eeq
  Let $m_1$, $m_2\in \{1, \dots, d_{\rm max}\}$. We write
  \bea
  \mathfrak{I}_\phi^{(m_1,m_2)} &=& \{\iota \in \mathfrak{I} | \iota_j = 0
  \Leftrightarrow \phi_j=0, \iota_{i_1(\phi)}=m_1, \iota_{i_2(\phi)}=m_2\}, \nn \\
    \mathfrak{I}_\phi^{(m_1)} &=&
   \{\iota \in \mathfrak{I} | \iota_j = 0
  \Leftrightarrow \phi_j=0, \iota_{i_1(\phi)}=m_1, \iota_{i_2(\phi)}>d_{\rm
    max}-m_1 \},  \nn \\
    \mathfrak{I}_\phi^{(>d_{\rm max})} &=&
   \{\iota \in \mathfrak{I} | \iota_j = 0
  \Leftrightarrow \phi_j=0, \iota_{i_1(\phi)}>d_{\rm max}\},
  \eea
  and we write $\aleph$ to indicate any of the superscripts $(m_1, m_2)$,
  $(m_1)$, or $(>d_{\rm max})$ where $1< m_1+m_2\leq d_{\rm max}$ and $1\leq m_1
  \leq d_{\rm max}$. Clearly, we have
  \beq
  \mathfrak{I} = \bigsqcup_{\phi,\aleph} \mathfrak{I}_\phi^\aleph.
  \eeq
  \begin{thm}
    There exists a choice of $u_j^\iota\in \bbQ+\ri \bbQ$, $\iota\in\mathfrak{I}_\phi^\aleph$, such that
    \ben[i.]
    \item \cref{prob:e8num} is equivalent to solving, for every $(\phi, \aleph)$, the linear system over
      the rationals
      \beq
      \sum_{\kappa\in \mathfrak{I}_{\phi}^{\aleph}} (\cA_\phi^\aleph)_{\iota
        \kappa} N_{\kappa}^{(k)} = (\cB_\phi^\aleph)_\iota^{(k)},
      \label{eq:numsys}
      \eeq
      where
      \ben[(a)]
      \item 
      \beq
      (\cA_\phi^\aleph)_{\iota \kappa} = \mathfrak{Re}\l(\prod_j
      (u_j^\iota)^{\kappa_j} \r),
      \eeq
      as in {\bf Method 1};
      \item $(\cB_\phi^\aleph)_\iota^{(k)}$ is computed from
      $Q-$derivatives of order at most $d_{\rm max}$ of $\chi_{\wedge^k \mathfrak{e}_8}$ at points
      $Q^\iota\in\cT_{E_8}$ such that $\chi_{j}(Q^\iota)=u_j^\iota$, as
        in {\bf Method 2};
      \een          
      \item the pre-images $Q^\iota=Q(u^\iota)$ can be obtained
          numerically via {\bf  Method 3} to precision $2^{-\mathsf{M}_0}$,
          $\mathsf{M}_0\in\bbN$, and there exist {\rm a priori} analytical
          bounds on the resulting numerical error on
        $(\cB_\phi^\aleph)_\iota^{(k)} \in \bbQ$ at the points
          $Q^\iota=Q(u^\iota)$ such that for sufficiently large $\mathsf{M}_0$, the numerical calculation of
        $(\cB_\phi^\aleph)_\iota^{(k)}$ via {\bf  Method 3} can be rounded
          rigorously to its {\rm exact}
        rational value;
          \item the system \eqref{eq:numsys} is non-degenerate, of rank lower
        than $r_{\rm max}$ for all $(\phi, \aleph)$, and can be solved
        effectively         using $p$-adic expansions as in {\bf Method 4}.
    \een

    \label{thm:mainthm}
    
  \end{thm}
  
%

  The next Section gives a detailed account of both the proof of
  \cref{thm:mainthm} and its concrete computer implementation, leading to the
  calculation of all $N_\iota^{(k)}$ in \eqref{eq:e8num2}. The reader who is more
  interested in the applications of the solution of \cref{prob:e8num} via
  \cref{thm:mainthm} may wish to skip the next section and move on directly to \cref{sec:appl}.

  \section{Proof of \cref{thm:mainthm}}
  \label{sec:computeapp}
  \subsection{Partition of the monomial set}
\label{sec:partition}

The key point we explained at the end of the previous section is that while
{\it generic} choices of sampling sets in {\bf Method 1} give rise to large, dense linear
systems, there are {\it special} choices of the numerical values of
$Q^{(\iota)}\in \cT_{E_8}$ that
reduce the calculation of $N_\iota^{(k)}$ to the solution of a large number of
linear problems of much smaller size. We explore this idea in detail in this section.

Let
\beq
\bary{ccccl}
\varphi & :  & \mathfrak{I} & \to & \{0,1\}^8 \\
& & \iota & \to & (\theta(\iota_1), \dots, \theta(\iota_8))
\eary
\eeq
where for $n \in \bbN$
\beq
\theta(n)=\left\{\bary{cc} 1 & n\neq 0 \\ 0 & n=0 \eary \right.,
\eeq
and let $\mathfrak{I}_{\phi} = \varphi^{-1}(\phi)$ for $\phi\in
\{0,1\}^8$. The map $\varphi$ maps an 8-tuple of admissible exponents
$(\iota_1, \dots, \iota_8)$ into an 8-tuple $(\varphi_1(\iota), \dots,
\varphi_8(\iota))$ with $\varphi_j(\iota)=\theta(\iota_j)$ equal to one if $\iota_j$ is
non-zero, and zero otherwise.
Its fibres $\mathfrak{I}_{\phi}$
  partition the monomial set into $2^8$ subsets of cardinality varying from 1
  to $\approx 3\times 10^4$,
  \beq
  \bigsqcup_{ \phi} \mathfrak{I}_{ \phi} = \mathfrak{I}.
  \eeq
  We furthermore define a strict weak ordering on the set
  of admissible exponents as
  \beq
  \iota < \kappa ~ \Longleftrightarrow ~ |\varphi(\iota)| <
  |\varphi(\kappa)|,
  \label{eq:preord0}
  \eeq
  and for subsets $\mathfrak{J}$, $\mathfrak{K}$ we will write
  $\mathfrak{J}<\mathfrak{K}$ if $\iota<\kappa$,
  $\forall~\iota\in\mathfrak{J}$, $\kappa\in\mathfrak{K}$.

  The crucial point is now as follows: consider a set of points on the torus $(Q^{(\iota)}_l)_l\in
  \cT_{E_8}$, $\iota \in \mathfrak{I}$,
  such that
  %
  \beq
  \chi_{j}(Q^{(\iota)})~\in~\varphi_j(\iota)~
  \bbC.
  \label{eq:admQ0}
  \eeq
  Evaluating \eqref{eq:e8num2} at $Q^{(\iota)}$ we have that only the coefficients
  $N^{(k)}_{\kappa}$ with $\varphi_k (\kappa) \leq \varphi_k (\iota)$ will appear
  on the r.h.s. of \eqref{eq:e8num2}. Then the linear problems
  \beq
\mathfrak{L}_{\phi} \triangleq  \l\{  \chi_{\wedge^k
  \mathfrak{e}_8}(Q^{(\iota)})- \sum_{\stackrel{\kappa < \iota}{\varphi_l (\kappa) \leq \varphi_l (\iota)}} N^{(k)}_{\kappa} \chi_{1}^{\kappa_1} \dots  \chi_{8}^{\kappa_8} = \sum_{\kappa
      \in \mathfrak{I}_{\varphi(\iota)}} N^{(k)}_{\kappa} \chi_{1}^{\kappa_1} \dots
\chi_{1}^{\kappa_8}\r\}_{\stackrel{\iota\in\mathfrak{I}}{\varphi (\iota) = \phi}}
\label{eq:Lphi}
  \eeq
are fully determined by the solution of $\mathfrak{L}_{\phi'}$ for $|\phi'| < |\phi|$, are each of size
$|\mathfrak{I}_{\phi}|$, and their solution for all $\phi$ solves
\cref{prob:e8num}: the end result is to break up our original linear system into $2^8=256$
linear subsystems of size varying from 1 to $\approx 3 \times 10^4$. \\

\subsubsection{Refining the partition}
The original linear problem for generic sampling values as described in {\bf
  Method 1} is theoretically solved in
$\cO(|\mathfrak{I}|^3)$ time; the special choice \eqref{eq:admQ0} breaks it up to
  $\approx 10^2$ sub-systems of size at most $\approx 10^{-2} |\mathfrak{I}|$,
giving an
improvement of a factor more than $10^4$ in the expected
runtime using floating-point precision. However 
this isn't satisfactory yet, since
%
we are still off by one order of magnitude in the size of the individual
  systems we can practically solve,
  \beq
  \mathrm{sup}_\phi \mathrm{rank} \mathfrak{L}_{\phi} \gg
   r_{\rm max}.
  \eeq

  In order to reduce $\mathfrak{L}_{\phi}$ to sub-systems of smaller size we
  refine the partition of $\mathfrak{I}$
  by considering proper subsets of
  $\mathfrak{I}_\phi$ having two fixed exponents
  equal to, or greater than, an integer less than $d_{\rm max}$. To this aim 
  we introduce maps $\vartheta_{ \phi}^{(1)}$ and $\vartheta_{\phi}^{(2)}$, with
  \beq
  \bary{ccccl}
  \vartheta_{\phi,(l)} & : & \mathfrak{I}_{\phi} & \to & \bbN_0, \qquad l=1,2, \\
  & & \iota & \to & \iota_{\sigma( \phi,l)}
  \eary
  \eeq
  where $\sigma( \phi,l)$ is defined by
  \bea
  \dim V_{\omega_{\sigma( \phi,1)}} &=& \mathrm{max} \{\dim  V_{\omega_j}\}_{\phi_j \neq 0}, \nn
  \\
  \dim V_{\omega_{\sigma( \phi,2)}} &=&   \mathrm{max} \{\dim  V_{\omega_j}\}_{\phi_j \neq 0, j\neq \sigma( \phi,1)}.
  \eea
  %
  %
  They can be used to refine the strict weak ordering \eqref{eq:preord0} on
  $\mathfrak{I}$ as follows
  \beq
  \iota < \kappa ~ \Longleftrightarrow ~
  \l\{
  \bary{cc}
  |\varphi(\iota)| <
  |\varphi(\kappa)|, & \mathrm{or} \\
  \varphi(\iota)=\varphi(\kappa)~\mathrm{and}~ \iota_{\sigma( \phi,1)}\leq
  d_{\rm max}, \kappa_{\sigma( \phi,1)}> d_{\rm max}, &  \mathrm{or} \\
  \varphi(\iota)=\varphi(\kappa)~\mathrm{and}~ \iota_{\sigma( \phi,1)},  \iota_{\sigma( \phi,2)}, \kappa_{\sigma( \phi,1)}\leq
     d_{\rm max}, \kappa_{\sigma( \phi,2)}> d_{\rm max}. &
  \eary
  \r.
  \label{eq:preord1}
  \eeq
  Let $m_l \in \bbN$, $l=1,2$ with $m_1+m_2\leq d_{\rm max}=5$, and introduce for fixed
  $\phi$ the following 16 subsets of $\mathfrak{I}_{\phi}$:
  \bea
  \mathfrak{I}^{(m_1,m_2)}_{\phi} & \triangleq & \vartheta_{\phi,1}^{-1}(m_1) \cap
  \vartheta_{\phi,2}^{-1}(m_2), 
  \\
  \mathfrak{I}^{(m_1)}_{\phi} & \triangleq & \vartheta_{\phi,1}^{-1}(m_1) \cap
  \vartheta_{\phi,2}^{-1}(\{m_2>d_{\rm max}-m_2\}),
  \\
  \mathfrak{I}^{(>d_{\rm max})}_{\phi} & \triangleq &
  \vartheta_{\phi,1}^{-1}(\{ m_1> d_{\rm max}\}).
  \eea
  Clearly,
  \beq
\mathfrak{I}= \bigsqcup_{\phi} \bigsqcup_{m_1+m_2\leq d_{\rm max}}
\mathfrak{I}^{(m_1,m_2)}_{\phi} \bigsqcup_{1\leq   m_1\leq d_{\rm
    max}}\mathfrak{I}^{(m_1)}_{\phi} \sqcup \mathfrak{I}^{(>d_{\rm
    max})}_{\phi}
\label{eq:partI}
  \eeq
  and
  \beq
  \iota < \kappa < \lambda \quad \forall~ \iota\in \mathfrak{I}^{(m_1,m_2)}_{\phi},
  \kappa\in \mathfrak{I}^{(m_1)}_{\phi}, \lambda \in  \mathfrak{I}^{(>d_{\rm max})}_{\phi}.
  \eeq
  We now have
  \beq
  \sup\Big\{\big|\mathfrak{I}^{(m_1,m_2)}_{\phi}\big|,\big|\mathfrak{I}^{(m_1)}_{\phi}\big|,
  \big|\mathfrak{I}^{(>d_{\rm max})}_{\phi}\big|\Big\}_{m_1,m_2,\phi} =
  \Big|\mathfrak{I}^{(>d_{\rm max})}_{(1,1,0,0,1,1,1,1)}\Big|=3027\ll r_{\rm max},
  \eeq
so if the linear problem $\mathfrak{L}_\phi$ in
\eqref{eq:Lphi} could be reduced down to linear sub-systems of rank
$|\mathfrak{I}^{(m_1,m_2)}_{\phi}\big|$, $|\mathfrak{I}^{(m_1)}_{\phi}|$
and $|\mathfrak{I}^{(>d_{\rm max})}_{\phi}|$, these could now be individually
solved explicitly. This can be done by mimicking \eqref{eq:Lphi}:
for $\iota\in \mathfrak{I}_\phi$, consider $(Q^{(\iota)}_l)_l\in \cT_{E_8}$
such that
\beq
\chi_{j}(Q^{(\iota)}) = u_j^{(\iota)}  \in  \varphi_j(\iota)
\l\{
\bary{cc}
\delta_{j,\sigma(\phi,1)}\delta_{j,\sigma(\phi,2)}, &
\iota\in\mathfrak{I}_\phi^{(m_1,m_2)}, \\
\delta_{j,\sigma(\phi,1)}, &
\iota\in\mathfrak{I}_\phi^{(m_1)}, \\
1, &
\iota\in\mathfrak{I}_\phi^{(> d_{\rm max})},
\eary
\r\}
~ \bbC.
\label{eq:admQ1}
\eeq
and define
%
  \bea
\mathfrak{L}_{\phi}^{(m_1,m_2)} & \triangleq &  \Bigg\{  \frac{1}{(m_1)!
  (m_2)!}\frac{\de^{m_1+m_2}}{\de \chi_{{\sigma(\phi,1)}}^{m_1} \de\chi_{{\sigma(\phi,2)}}^{m_2}}\chi_{\wedge^k
  \mathfrak{e}_8}(Q^{(\iota)})- \sum_{\kappa\in
  \mathfrak{P}_{\phi}^{(m_1, m_2)}}
N^{(k)}_{\kappa}
\prod_{\stackrel{j\neq \sigma(\phi,l)}{l=1,2}} \chi_{j}^{\kappa_j}\big(Q^{(\iota)}\big)
\nn \\ &=& \sum_{\kappa
      \in \mathfrak{I}_{\phi}^{(m_1,m_2)}} N^{(k)}_{\kappa}
\prod_{\stackrel{j\neq \sigma(\phi,l)}{l=1,2}} \chi_{j}^{\kappa_j}\big(Q^{(\iota)}\big)\Bigg\}_{\iota\in\mathfrak{I}_{\phi}^{(m_1,m_2)}},
\label{eq:Lphiref1}
\\
\mathfrak{L}_{\phi}^{(m_1)} & \triangleq &  \Bigg\{
\frac{1}{(m_1)!}\frac{\de^{m_1} \chi_{\wedge^k
  \mathfrak{e}_8}}{\de
  \chi_{{\sigma(\phi,1)}}^{m_1}}\big(Q^{(\iota)}\big)- \sum_{\kappa \in
  \mathfrak{P}_{\phi}^{(m_1)}}
  N^{(k)}_{\kappa}
\prod_{j\neq \sigma(\phi,1)} \chi_{j}^{\kappa_j}\big(Q^{(\iota)}\big)
\nn \\
&=& \sum_{\kappa \in \mathfrak{I}_{\phi}^{(m_1)}} N^{(k)}_{\kappa}
\prod_{j\neq \sigma(\phi,1)} \chi_{j}^{\kappa_j}\big(Q^{(\iota)}\big)\Bigg\}_{\iota\in\mathfrak{I}_{\phi}^{(m_1)}},
\label{eq:Lphiref2}
\\
\mathfrak{L}_{\phi}^{(>d_{\rm max})} & \triangleq &  \Bigg\{
\chi_{\wedge^k
  \mathfrak{e}_8} \big(Q^{(\iota)}\big)- \sum_{\kappa \in
  \mathfrak{P}_{\phi}^{(>d_{\rm max})}}
  N^{(k)}_{\kappa}
\prod_{j} \chi_{j}^{\kappa_j}\big(Q^{(\iota)}\big)
= \sum_{\kappa \in \mathfrak{I}_{\phi}^{(>d_{\rm max})}} N^{(k)}_{\kappa}
\prod_{j}
\chi_{j}^{\kappa_j}\big(Q^{(\iota)}\big)\Bigg\}_{\iota\in\mathfrak{I}_{\phi}^{(>d_{\rm
  max})}},
\label{eq:Lphiref3}
\eea
where
\bea
\mathfrak{P}_{\phi}^{(m_1, m_2)} & := & \l\{ \kappa \in \mathfrak{I} |
\kappa < \iota,~
\kappa_{\sigma(\phi,l)}=m_l
~\forall l=1,2, \iota \in \mathfrak{I}_{\phi}^{(m_1, m_2)} \r\}, \nn \\
\mathfrak{P}_{\phi}^{(m_1)} & := & \l\{ \kappa \in \mathfrak{I} |
\kappa < \iota,~ \kappa_{\sigma(\phi,1)}=m_1
~\forall~ \iota \in \mathfrak{I}_{\phi}^{(m_1)} \r\} \bigcup_{m_2=1}^{d_{\rm max}-m_2}
  \mathfrak{I}_\phi^{(m_1,m_2)}, \nn \\
\mathfrak{P}_{\phi}^{(>d_{\rm max})} & := & \l\{ \kappa \in \mathfrak{I} |
\kappa < \iota,
\iota \in \mathfrak{I}_{\phi}^{(>d_{\rm max})} \r\} \bigcup_{m_1+m_2\leq d_{\rm max}}
  \mathfrak{I}_\phi^{(m_1,m_2)} \bigcup_{m_1=1}^{d_{\rm max}}
  \mathfrak{I}_\phi^{(m_1)}. \nn \\
\eea
Since $\mathfrak{P}_{\phi}^{(m_1, m_2)}\subsetneq \mathfrak{P}_{\phi}^{(m_1)}
\subsetneq \mathfrak{P}_{\phi}^{(>d_{\rm max})}$, we have that
$\{\mathfrak{L}_{\phi}^{(m_1, m_2)}\}_{m_1,m_2}$, $\{\mathfrak{L}_{\phi}^{(m_1)}\}_{m_1}$ and
$\mathfrak{L}_{\phi}^{(>d_{\rm max})}$ can be solved in this order to give
inhomogeneous linear systems of rank $\{|\mathfrak{I}_{\phi}^{(m_1, m_2)}|\}_{m_1,m_2}$,
$\{|\mathfrak{I}_{\phi}^{(m_1)}|\}_{m_1}$ and $|\mathfrak{I}_{\phi}^{(>d_{\rm max})}|$
respectively. \\

Let us summarise what we have done so far in this section: the original
\cref{prob:e8num} is equivalent to the $16\times 256=4096$ numerical inhomogeneous systems of linear equations
\beq
\bary{rcrcl}
\mathfrak{L}_\phi^{(m_1,m_2)} & : & \sum_{\kappa\in\mathfrak{I}_\phi^{(m_1,m_2)}}
\l(\cA_\phi^{(m_1,m_2)}\r)_{\iota \kappa} N^{(k)}_{\kappa} & = &
\l(\cB_\phi^{(m_1,m_2)}\r)^{(k)}_\iota, \\
\mathfrak{L}_\phi^{(m_1)} & : & \sum_{\kappa\in\mathfrak{I}_\phi^{(m_1)}}
\l(\cA_\phi^{(m_1)}\r)_{\iota \kappa} N^{(k)}_{\kappa} & = &
\l(\cB_\phi^{(m_1)}\r)^{(k)}_\iota, \\
\mathfrak{L}_\phi^{(>d_{\rm max})} & : &
\sum_{\kappa\in\mathfrak{I}_\phi^{(>d_{\rm max})}}
\l(\cA_\phi^{(>d_{\rm max})}\r)_{\iota \kappa} N^{(k)}_{\kappa} & = &
\l(\cB_\phi^{(>d_{\rm max})}\r)^{(k)}_\iota,
\eary
\label{eq:ABN}
\eeq
where, writing $\aleph$
to indicate any of the symbols $(m_1, m_2)$, $(m_1)$, or $(>d_{\rm max})$,  the matrices $\cA_\phi^\aleph$ (resp. $\cB_\phi^\aleph$) are computed
from the r.h.s. (resp. l.h.s.) of
  \eqref{eq:Lphiref1}--\eqref{eq:Lphiref3} with $\chi_{j}$ (resp.  $\chi_{\wedge
    \mathfrak{e}_8}$ and $\chi_{j}$) evaluated at
  points $Q^{(\iota)}\in \cT_{E_8}$ satisfying \eqref{eq:admQ1}. The
  inhomogeneous piece $(\cB_\phi^\aleph)^{(k)}_\iota$ depends, by
  \eqref{eq:Lphiref1}--\eqref{eq:Lphiref3}, on all $N^{(k)}_{\iota'}$ with
  $\iota'<\iota$: therefore, if $\mathfrak{J}\subset \mathfrak{I}$,
  $\mathfrak{K}\subset \mathfrak{I}$ are any two elements in
  the partition \eqref{eq:partI} of $\mathfrak{I}$, the corresponding
  linear problems in \eqref{eq:ABN} must be solved sequentially if $\mathfrak{J}<\mathfrak{K}$,
  but they can be solved in parallel if  $\mathfrak{J}$ is incomparable with
  $\mathfrak{K}$ under the weak order $<$. Furthermore, the determination and
  solution of \eqref{eq:ABN} is
  compatible with the computability bounds \eqref{eq:bounds}: writing out
  $\mathfrak{L}_\phi^\aleph$ only
  requires the calculation of $Q$-derivatives of $\chi_{\wedge
    \mathfrak{e}_8}$ up to order $d_{\rm max}$, and solving it involves
  inverting linear operators $\cA_\phi^\aleph$  of rank $\ll r_{\rm max}$. Moreover,
since the latter does not depend on $k$, the inversion solves \eqref{eq:ABN} for all $k$ in a
  single go.

  \subsection{The computation of $\mathfrak{L}_\phi^\aleph$}
\label{sec:complphi}
  Let us show how the $\mathfrak{L}_\phi^\aleph$ can be computed in
  practice. Let $u^{(\iota)}\in\bbC^8$ be as in \eqref{eq:admQ1} and let
  $Q^{(\iota)}=Q(u^{(\iota)})$ be any pre-image of $u^{(\iota)}$.  Given $u^{(\iota)}$, the matrices $(\cA_\phi^\aleph)_{\iota\kappa}$ are
  immediately computed as
  \beq
  (\cA_\phi^\aleph)_{\iota\kappa} = \prod_j \l(u_j^{(\iota)}\r)^{\kappa_j}.
  \label{eq:Aphi}
  \eeq
  The calculation of $(\cB_\phi^\aleph)_\iota^{(k)}$ is however significantly
  more involved, as it requires to compute $\chi_{j}-$ derivatives of
  $\chi_{\wedge \mathfrak{e}_8}(Q)$ up to order $d_{\rm max}$ at
  $Q=Q^{(\iota)}$ for all $k$. We describe how this is done in the rest of
  this section.
  
  \subsubsection{Step 1: $Q$-derivatives from $\Gamma(\omega_j)$}

  The first step is to record
  the value of all $Q$-derivatives of $\chi_{\wedge^k \mathfrak{e}_8}$ at
  $Q=Q(u^{(\iota)})$ for $\iota\in\mathfrak{I}_\phi^{\aleph}$ up to a sufficiently high order. To this aim, write
  \beq
  \gamma_{\aleph} = \left\{\bary{cc} m_1+m_2, &  \aleph=(m_1,m_2),\\
  m_1, & \aleph=m_1, \\
  0, & \aleph>d_{\rm max}. \eary \right.
  \eeq
 Let
  \beq
  \mathscr{C}_{k}=\l\{c \in (\bbZ^+_0)^8 \bigg| \sum_i c_i = k\r\}
  \eeq
  be the set of compositions of $k\in \bbZ^+$ of length eight. Consider the
  total order on $\mathscr{C}_{k}$ obtained by sorting ascendingly with respect to the values of the components $c_1,
  \dots, c_8$ taken in this order, and denote $\mathscr{C}_{k}^i$, $i=1, \dots,
  |\mathscr{C}_{k}|=\binom{k+7}{7}$ for the $i^{\rm th}$ element of
  $\mathscr{C}_{k}$ as a totally ordered set. Then what we want to compute,
  for any $c\in\mathscr{C}_{\gamma_\iota}$, is
  \beq
  \DD(c,\iota, V) \triangleq D^c  \chi_{V} (Q(u^{(\iota)}))=
\frac{\de^{|c|}  \chi_{V}}{\de Q_1^{c_1} \dots \de Q_8^{c_8}}(Q(u^{(\iota)})) = \sum_{w\in
    \Gamma(V)} \prod_i (w_i-c_i+1)_{c_i}
Q_i^{w_i-c_i}\bigg|_{Q=Q(u^{(\iota)})},
\label{eq:derQ0}
  \eeq
  for $V=\wedge^k\mathfrak{e}_8$ and $V=V(\omega_j)$, where we wrote $(a)_n$
  for the Pochhammer symbol $\Gamma(a+n)/\Gamma(a)$.

  The sums in \eqref{eq:derQ0} are unwieldy in that form, since we have up to
  $|\Gamma_{\rm red}(V_{\omega_3})|=186481$ summands (weight spaces not
  counted with multiplicties) when computing \eqref{eq:derQ0} for the
  fundamental representations
  $V=V(\omega_j)$, and up to $|\Gamma_{\rm
    red}(V_{\wedge^{120}\mathfrak{e}_8})|\approx 10^{17}$ for the exterior
  powers of the adjoint, $V=\wedge^{k}\mathfrak{e}_8$. It
  is however possible to reduce the calculation of \eqref{eq:derQ0} to sums
  having at most $|\Gamma_{\rm red}(\omega_1)|=2401$, $|\Gamma_{\rm
    red}(\omega_7)|=|\Delta^+ \cup \Delta^- \cup \mathbf{0}|=241$, or $|\Gamma_{\rm red}(\omega_8)|=26401$ terms, as follows. To compute
  \eqref{eq:derQ0} for $V=\wedge^{k}\mathfrak{e}_8$, we first evaluate the virtual power sum characters
  \beq
  \tilde{\DD}_{k,j}(c,\iota) \triangleq
  \frac{\de^{\gamma_\iota}  \chi_{\mathfrak{e}_8}}{\de Q_1^{c_1} \dots \de
    Q_8^{c_8}}(Q^k(u^{(\iota)})) = \sum_{w\in \Gamma(\omega_j)} \prod_i (w_i-c_i+1)_{c_i}
Q_i^{k w_i-c_i}\bigg|_{Q=Q(u^{(\iota)})},
\label{eq:derQpowk}
  \eeq
  and then recursively compute \eqref{eq:derQ0} using Newton identities:
  \beq
  \DD(c,\iota, \wedge^k V(\omega_j)) =  \sum_{m=1}^k (-1)^{m+1}
  \sum_{c'+c''=c}  \l(\prod_l \binom{c_l}{c'_l}\r)   \tilde{\DD}_{m,j}(c',\iota)
  \DD(c'',\iota, \wedge^{k-m} V(\omega_j)).
\label{eq:derQwedk}
  \eeq
  for $j=7$. To evaluate the fundamental characters $\DD(c,\iota, V(\omega_i))$, we compute
  directly from \eqref{eq:derQ0} for the three smallest-dimensional fundamental representations
  corresponding to $i=1,7,8$; for the remaining five, we make use of the
  following identities in $\mathrm{Rep}(E_8)$, which can be easily proved in the
  same vein of \cref{ex:dec}:
  \bea
  \label{eq:Vom6}
  V_{\omega_6} &=& \wedge^2 \mathfrak{e}_8 \ominus \mathfrak{e}_8,\\
  V_{\omega_5} &=& \wedge^3 \mathfrak{e}_8 \oplus \mathfrak{e}_8 \ominus
  \mathfrak{e}_8 \otimes \mathfrak{e}_8,\\
  V_{\omega_4} &=& \wedge^4 \mathfrak{e}_8 \oplus \wedge^3 \mathfrak{e}_8\oplus \wedge^2 \mathfrak{e}_8 \ominus
  \wedge^2 \mathfrak{e}_8\otimes \mathfrak{e}_8 \ominus V_{\omega_5},\\
    V_{\omega_3} &=& \wedge^5 \mathfrak{e}_8 \oplus 2 \wedge^4
    \mathfrak{e}_8\oplus 2 \wedge^2 \mathfrak{e}_8 \ominus \mathfrak{e}_8 \ominus
  \wedge^3 \mathfrak{e}_8\otimes \mathfrak{e}_8 \ominus V_{\omega_4},\\
    V_{\omega_2} &=& \wedge^2 V(\omega_1) \oplus V_{\omega_1}-
    V_{\omega_1}\otimes\mathfrak{e}_8 + V_{\omega_8}.
      \label{eq:Vom2}
  \eea
  Let now
  \beq
  m^{\rm max}_j = \l\{ \bary{cc} 2 & j=1, \\ 5 & j=7, \\ 1 & j=8. \eary \r.
  \eeq
 $\DD(c,\iota, V(\omega_j))$, $j=2,\dots, 6$ can then be calculated by
  \eqref{eq:Vom6}--\eqref{eq:Vom2} knowing $\DD(c,\iota,
  \wedge^{m^{\rm max}_j}
  V(\omega_j))$ with $j=1,7,8$; and the latter is computed by
  \eqref{eq:derQ0}--\eqref{eq:derQwedk} as an explicit polynomial in
  $\tilde{\DD}_{k,j}(c,\iota)$, $k=1, \dots, m^{\rm max}_j$,  for the same
  values of $j$. All in all, we have proved that there exist explicit polynomials over the
  rationals such that
  \bea
  \DD(c,\iota, V(\omega_j)) & \in & \bbQ\l[\l\{ \tilde{\DD}_{1,1}(c,\iota),\tilde{\DD}_{2,1}(c,\iota),\{
       \tilde{\DD}_{l,7}(c,\iota)\}_{l=1}^{5},
       \tilde{\DD}_{1,8}(c,\iota)\r\}_{c\in\mathscr{C}_{\delta_i}}\r], \nn \\
    \DD(c,\iota, \wedge \mathfrak{e}_8) & \in &
    \bbQ\l[\l\{\tilde{\DD}_{l,7}(c,\iota)\r\}_{\stackrel{l=1,\dots,120}{c\in\mathscr{C}_{\delta_i}}}\r].
        \label{prop:step1}
  \eea
\subsubsection{Step 2: $\chi_{\omega}$-derivatives, graph expansions and the Fa\`a di Bruno formula}
\label{sec:faa}

The second step consists of computing $\chi_{j}$-derivatives of
$\chi_{\wedge \mathfrak{e}_8}(g)$ at $\chi_{j}(g)=u^{(\iota)}$ from knowledge
of its $Q$-derivatives at $Q=Q(u^{(\iota)})$. To start with, let
\beq
\cJ_{ij} \triangleq \l(\frac{\de \chi_{i}}{\de Q_j}\r) \in \mathrm{Mat}\big(8,\bbC[Q^\pm]\big)
\label{eq:jac}
\eeq
denote the Jacobian matrix of the fundamental characters $\chi_{i}$ with respect
to the exponentiated linear co-ordinates $Q_j$, which we assume here to be
non-singular at $Q(u^{(\iota)})$ for all $\iota$. Let again
$\iota\in\mathfrak{I}$, suppose $\gamma_\iota >0$ and let
$c\in\mathscr{C}_{\gamma_\iota-1}$. What we want to compute is the numerical
matrix
\beq
\cJ^{\rm inv}_{c,\iota} \triangleq \frac{\de^{|c|} \cJ^{-1}}{\de Q_1^{c_1}\dots
  \de Q_8^{c_8}}\Bigg|_{Q=Q(u^{(\iota)})}.
\eeq
Iterated differentiation of
\beq
\frac{\de \cJ^{-1}}{\de Q_j} = -  \cJ^{-1} \frac{\de \cJ}{\de Q_j}  \cJ^{-1}.
\eeq
leads to the following easy
\begin{prop}
  \label{prop:dJinv}
  We have
\beq
\cJ^{\rm inv}_{c,\iota} = \sum_{r=1}^{|c|} \sum_{d^{(1)}+\dots d^{(r)} = c} (-1)^r \cJ^{\rm
  inv}_{0,\iota} \prod_{l=1}^r
\cJ_{d^{(r)},\iota} ~\cJ^{\rm inv}_{0,\iota},
\label{eq:Jinviota}
\eeq
where we defined for $d\leq c \in\mathscr{C}_{\gamma_\iota-1}$
\bea
\cJ_{d,\iota} &\triangleq& \frac{\de^{|d|} \cJ}{\de Q_1^{d_1}\dots \de
  Q_8^{d_8}}\Bigg|_{Q=Q(u^{(\iota)})}.
\eea
\end{prop}
%
%
On the r.h.s. of \eqref{eq:Jinviota}, $\cJ^{\rm inv}_{0,\iota}$ is computed from
$\DD(d,\iota,V(\omega_j))$ with $|d|=1$; the latter gives the numerical expression
for the Jacobian matrix $\cJ$ at $Q=Q(u^{(\iota)})$, and $\cJ^{\rm
  inv}_{0,\iota}$ is by definition its inverse. And as for
$\cJ_{d^{(r)},\iota}$, its entries are just given by values
$\DD(d,\iota,V(\omega_j))$ for $|d|=1, \dots, \gamma_\iota$: therefore
$\cJ^{\rm inv}_{c,\iota}$ is computed by \eqref{eq:Jinviota} as an explicit
polynomial expression in $\DD(d,\iota,V(\omega_j))$ and $\cJ_{0,\iota}^{\rm
  inv}$, which in turn are determined by $\tilde{\DD}_{k,j}(c,\iota)$, $k=1, \dots, m^{\rm
  max}_j$, $j=1,7,8$, all of which were computed in Step 1 above.


What we will do next, armed with $\cJ^{\rm inv}_{c,\iota}$ computed as in \eqref{eq:Jinviota}, is to
convert\footnote{Stated more invariantly, we want to look at the co-ordinate
  expression, in the chart induced by the exponentiated linear co-ordinates
  $Q$ on $\cT_{E_8}$, of $\mathfrak{D}^c$ as a map
  $J^{|c|}_{u^{(\iota)}} \cT_{E_8} \to \bbC$
  from the fibre at $Q=Q(u^{(\iota)})$ of the $|c|^{\rm th}$ jet bundle on
  $\cT_{E_8}$ to $\bbC$.} the differential
operator $\mathfrak{D}^c \triangleq \de^{|c|}_{\de \chi_{1}^{c_1} \dots \chi_{8}^{c_8}}$
  into a differential operator in $(Q_1, \dots, Q_8)$ at $Q=Q(u^{(\iota)})$;
  at the end of the day this will provide a presentation of the
   Fa\`a di Bruno formula for the partial derivatives of
   composite functions in several variables, which appears to be new in this
   form. We start by the following
   \begin{defn}
     \label{def:FdB}
     For $k_L,k_R, n\in\bbN$, an $n-$decorated {\rm Fa\`a di Bruno graph} $\mathsf{G}$ of order $(k_L,k_R)$ is a
     decorated, ordered, oriented graph $(V=V_L \cup V_R, E)$, satisfying
     \ben
      \item $|V_L|=k_L$, $|V_R|=k_R \leq k_L$;
      \item $V_L=\{v_1^{(L)},\dots, v_k^{(L)}\}$ with $v_i^{(L/R)} \geq
        v_j^{(L/R)} \Leftrightarrow i\leq j$, $v_i^{(L)}\leq v_j^{(R)} ~\forall i,j$;
      \item every vertex in $V_L$ ($V_R$) has exactly one outgoing (incoming)
        oriented edge attached to it;
        \item any $v\in V_R$ is a leaf;
     \item if $v_i^{(L)}$ is adjacent to $v_j^{(R)}$, $v_s^{(L)}$ is adjacent
       to $v_t^{(R)}$, and $i<s$, then $j<t$;
       \item if $v_i^{(L)}$ is adjacent to $v_j^{(L)}$, and $i>j$, then the
         attaching edge is oriented from $j$ to $i$;
      \item $i:V_L \to \{1, \dots, n\}$.
       \een
     \end{defn}
   A helpful way to visualise the definition above is the following (see
   \cref{fig:faa}): the datum of $\mathsf{G}$ 
   gives two sets of vertices, $V_L$ and $V_R$, which we can arrange in two
   vertical columns in the plane; the number of left vertices is greater than
   that of right vertices (Property 1,2). Each left vertex has exactly one outgoing arrow emanating
   from it (Property 3), which can be of either of two types: a vertical arrow ending on another left vertex higher up in the
   column (Property 6), or a horizontal arrow ending on a right vertex; horizontal arrows
   are not allowed to cross (Property 5). Left vertices are decorated with a
   positive integer less than or equal to $n$.

   We denote $\mathrm{FG}_{k_L,k_R,n}$ the set of $n$-decorated Fa\`a di Bruno graphs of order
   $(k_L,k_R)$. The graphs for $k_L=3$ are depicted in \cref{fig:faa}; in
   general, there are $k_L!$ un-decorated graphs with fixed value of $k_L$.

   \begin{figure}[t]
     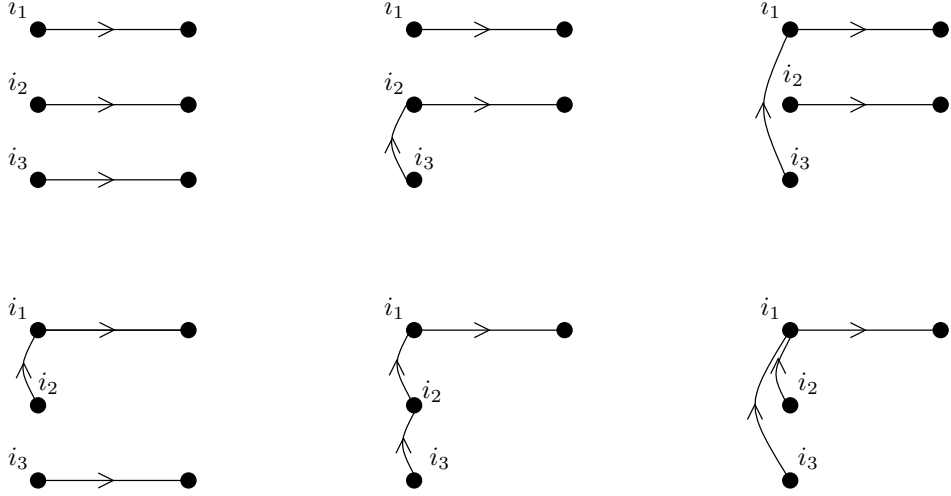
     \caption{The Fa\`a di Bruno graphs for $k_L=3$.}
     \label{fig:faa}
   \end{figure}

   \begin{prop}[Multi-variate Fa\`a di Bruno formula]
     \label{prop:FdB}
    Let $\delta>0$, $\mathsf{x}_0\in \bbR^n$, $k_L\in\bbN$, and $(f_l)_l \in C^{k_L}(B_\delta(x_0))$ ($l=1, \dots,
     n$), and suppose $\det\mathscr{J}\neq 0$ with $\mathscr{J}_{ij}=\de_{x_i}
     f_k(\mathsf{x}_0)$ in local co-ordinates $x_i$ around $\mathsf{x}_0$.
     Then, for $\epsilon:=\mathrm{sup}_{\mathsf{x}\in (B_\delta(x_0))}\|f(\mathsf{x})-f(\mathsf{x}_0) \| $
and $F \in C^{k_L}(B_\epsilon(f(x_0))$, we have
\beq
\frac{\de^{k_L}F}{\de f_{i_1}\dots \de f_{i_{k_L}}} = \sum_{\stackrel{k_R\leq k_L}{\mathsf{G}\in \mathrm{FG}_{k_L,k_R,n}}} \sum_{j_1, \dots, j_{k_L}=1}^n \prod_{r=1}^{k_L}
   \l(\prod_{v_l^{(L)}\in \mathrm{In}(v_r^{(L)})} \de_{x_{j_l}}\r) \mathscr{J}^{-1}_{i_r, j_r}
   \prod_{s=1}^{k_R} \l(\prod_{v_m^{(L)}\in \mathrm{In}(v_s^{(R)})}
   \de_{x_{j_m}} (F \circ f)\r)
   \label{eq:FdB}
   \eeq
where for $v\in V(\mathsf{G})$, $\mathsf{G}\in\mathrm{FG}_{k_L,k_R,n}$, we
indicate by $\mathrm{In}(v)\subset V_L$ as
the set of tails of arrows having $v$ as their head.
   
   \end{prop}

   \begin{proof}
    The claim follows by keeping track of the combinatorics of the Leibniz
    formula for the composition of linear differential operators of the form
    $\de_{f_i}=\sum_j (\mathscr{J})^{-1}_{ij}\de_{x_j}$: this can be encoded into
    graphical rules that boil down to the properties of Fa\`a di Bruno graphs
    in \cref{def:FdB}. Consider the order
    $k_L$ differential operator
    \beq
    \frac{\de^{k_L}}{\de f_{i_1}\dots \de
      f_{i_{k_L}}}\Bigg|_{f(\mathsf{x}_0)}=  \prod_{l=1}^{k_L}\frac{\de}{\de
      f_{i_l}}\Bigg|_{f(\mathsf{x}_0)}=\sum_{j_1,\dots j_{k_L}}
    \tilde{\mathscr{J}}_{i_1,j_1, \dots, i_{k_L}, j_{k_L}},
    \label{eq:sumscrJ}
    \eeq
    where
    \beq
       \tilde{\mathscr{J}}_{i_1,j_1, \dots, i_{k_L}, j_{k_L}} \triangleq \overrightarrow{\prod_{l=1}^{k_L}}\l( \mathscr{J}^{-1}_{i_l j_l}\frac{\de}{\de
         x_{j_l}}\r)\Bigg|_{\mathsf{x}_0}
       \label{eq:scrJ}
      \eeq
 and the order of the product in \eqref{eq:scrJ} has been chosen from left
    to right (the sum over $j_l$ in \eqref{eq:sumscrJ} yet making any particular choice of ordering immaterial, as the
    l.h.s. is obviously symmetric in $i_1, \dots i_{k_L}$).  To each factor
    in the ordered product \eqref{eq:scrJ} we associate a vertex $v_l^{(L)}$, $l=1, \dots, k_L$ in the
    plane, and we align the set of vertices vertically and order them from
    bottom to top corresponding to the rightwards order in
    \eqref{eq:scrJ}. Then for  each factor in \eqref{eq:scrJ} is a linear
    differential operator $ \mathscr{J}^{-1}_{i_l j_l}\frac{\de}{\de
         x_{j_l}}$ which can act on the factors to its right in either of the two following ways:
    \bit
    \item it can hit $\mathscr{J}^{-1}_{i_m j_m}$, $m>l$: in that case we draw an
      oriented edge from $v_l^{(L)}$ to $v_m^{(L)}$;
    \item it can move across all Jacobian factors to its right; in that case
      we draw a leaf $v_{\sigma_l}^{(R)}$ next to $v_{l}^{(L)}$ to its right, with an oriented
      edge from the latter to the former. The labels $\sigma_l$ for the
      right vertices $v_{\sigma_l}^{(R)}$ are chosen to respect the order of
      the left vertices and to take consecutive values from $1$ to $k_R$ for
      some positive integer $k_R \neq k_L$.
    \eit
    Each of these possiblities defines a Fa\`a di Bruno graph with a
    decoration of the left vertices by $i_l$, the row index of $\mathscr{J}^{-1}_{i_l,j_l}$. The Leibniz rule states that \eqref{eq:scrJ} reduces to a sum over all
    these possibilities, each giving an order $k_R$ differential operator
    whose coefficients are products of derivatives of $\mathscr{J}^{-1}$, with the orders of
    differentiation of the factors summing up to $k_L-k_R$. Plugging the resulting
    sum over $k_R$ and $\mathrm{FG}_{k_L, k_R,n}$ into \eqref{eq:sumscrJ} and
    acting on $F\in C^{k_L}(B_\epsilon(f(\mathsf{x}_0))$ gives \eqref{eq:FdB}.

   \end{proof}
  %
  %

   \begin{cor}
     We have
     \beq
     (\cB_\phi^\aleph)_\iota^{(k)}\in
     \bbQ\l[\l\{ \tilde{\DD}_{1,1}(c,\iota),\tilde{\DD}_{2,1}(c,\iota),\{
       \tilde{\DD}_{l,7}(c,\iota)\}_{l=1}^{\mathrm{max}(k,5)},
       \tilde{\DD}_{1,8}(c,\iota)\r\}_{c\in\mathscr{C}_{\delta_i}},\cJ_{0,\iota}^{\rm
         inv}\r]
     \eeq
     \label{cor:Bphi}
     \end{cor}

   This follows immediately from \cref{prop:FdB,prop:dJinv} and the
   description of Step~1 in \cref{prop:step1}, which
   give explicit expressions for the $\bbQ$-polynomials on the r.h.s.. It
   should be noted that $\cJ^{\rm inv}_{0,\iota}$ is also determined, albeit
   not polynomially, by the basic building blocks $\{\tilde{\DD}_{l,m_l^{\rm
       max}}(c,\iota)\}_{j=1,7,8}$. 
   
   \subsection{Computing $(\cB)_\phi^{(k)}$ exactly.}
\label{sec:compbphi}
There are still quite a few more hurdles yet to overcome in writing down
\eqref{eq:ABN}, as follows.

\ben
\item The main problem is that the whole method of the previous section hinges
  on finding an (arbitrary, but exact) pre-image $Q(u^\iota)$ of $\chi_{j}(Q)=u_j^\iota$
  in \eqref{eq:admQ1}. This requires
  solving a non-linear system of eight polynomial equations of high degree in
  eight variables, which cannot be achieved analytically but for very special values of $u_j^\iota$.
%
\item Of course, as in {\bf Method 3}, one could try to solve \eqref{eq:admQ1} for $Q^{(\iota)}=Q(u^{(\iota)})$
  numerically. The resulting linear problems \eqref{eq:ABN} will then have
  numerically approximated entries, and one can then hope to determine $N^{(k)}_\iota\in\bbZ$ by
  rounding up the numerical solution of \eqref{eq:ABN} to the nearest integer
  -- and with appropriate analytic bounds on the propagation of the numerical
  error, this rounding will be exact. But even then, the density and
  ill-conditioning of $\mathfrak{L}^\aleph_{\phi}$ rule out the use
  of floating-point arithmetics in its solution: the absolute value of the minimal principal value of
  $\cA_\phi^\aleph$ can be of order as little as $\approx 10^{-5000}$, meaning
  that tiny rounding errors in the numerical evaluation of $\chi_{\wedge^k
    \mathfrak{e}_8}$ lead to massive errors $\gg \cO(1)$ after solving for 
  $N^{(k)}_{\iota}$ in $\mathfrak{L}_{\phi}^\aleph$, making an exact
  rounding to the nearest integer unfeasible.
  \item Finally, one should be aware that the choice of points $Q^{(\iota)}\in
    \cT_{E_8}$ in \eqref{eq:admQ1} could potentially lead to a singular $\cA_\phi^\aleph$ in
    \eqref{eq:ABN}. This might appear to be an issue of lesser
    importance, as generic choices in \eqref{eq:admQ1} lead to non-degenerate
    linear problems; however, as we will see momentarily, to deal with (1)-(2) above we will have to
    consider a highly constrained choice of sampling values $Q^{(\iota)}$, so this
    point will become relevant as well.
  \een

  The main problem to deal with is that we need to be able to
 a) write and b) solve the linear problems \eqref{eq:ABN} exactly.  It turns
 out that {\bf Methods 2-4} can be used to achieve both of these
 objectives.

 \subsubsection{Exactly rounded floating point arithmetics}

 For the
  latter point, our strategy is to consider sampling values
  $(Q_l^{(\iota)})_l\in\cT_{E_8}$ such that
 \beq
\chi_{j}(Q^{(\iota)}) = u_j^{(\iota)}  \in  \varphi_j(\iota)
\l\{
\bary{cc}
\delta_{j,\sigma(\phi,1)}\delta_{j,\sigma(\phi,2)}, &
\iota\in\mathfrak{I}_\phi^{(m_1,m_2)}, \\
\delta_{j,\sigma(\phi,1)}, &
\iota\in\mathfrak{I}_\phi^{(m_1)}, \\
1, &
\iota\in\mathfrak{I}_\phi^{(> d_{\rm max})},
\eary
\r\}
~ \mathfrak{S}
\label{eq:admQ2}
\eeq
where $\mathfrak{S}\subset \l(\bbQ+\ri \bbQ\r)$. Then, since $N_\iota^{(k)}\in\bbZ$, the real part of
$\mathfrak{L}_\phi^\aleph$ becomes a linear system over $\bbQ$ (or $\bbZ$), which can
be solved using exact arithmetics, provided we can find an exact pre-image of
$u_j^{(\iota)}$ in \eqref{eq:admQ2} -- which as explained above we cannot do analytically in general. What we
will do instead will be to find an approximate solution $(\hat Q^{(\iota)},
\delta_{\hat Q^{(\iota)}})$ with $Q^{(\iota)}$ contained in a ball of radius
$\delta_{\hat Q^{(\iota)}}$ from $\hat Q^{(\iota)}$.
We will then compute from that a pair
$\big((\hat \cB_{\phi}^{\aleph})_\iota^{(k)}, \delta_{\hat \cB_{\phi}^{\aleph})_\iota^{(k)}}\big)$ out of which the exact
value of $(\cB_{\phi}^{\aleph})_{\iota}^{(k)} \in\bbQ+\ri \bbQ$ can be uniquely and exactly
determined.

We get started with the following

\begin{defn}
  Let $\mathsf{M}\in\bbN$. We denote
  \beq
  \mathrm{FPC}_{\mathsf{M}}=\l\{x\in\bbC \Big| \mathrm{Mant}_2(\mathfrak{Re}(x)) =
  \sum_{k=1}^M a_k 2^{-k}, \mathrm{Mant}_2(\mathfrak{Im}(x)) =
  \sum_{k=1}^M b_k 2^{-k} \r\},
  \label{eq:FPC}
  \eeq
  the set of complex floating-point numbers with $\mathsf{M}$-bit precision, where
  $a_k, b_k \in \{0,1\}$, and in \eqref{eq:FPC} we denoted
  \beq
  \mathrm{Mant}_2(x)= 2^{-\mathrm{Exp}_2(x)} x, \quad 
  \mathrm{Exp}_2(x) = \lceil \log_2 x \rceil.
  \eeq
  We shall also write
  \beq
  \varepsilon_x := 2^{\mathrm{Exp}_2(x)-\mathsf{M}}.
  \eeq
\end{defn}

\begin{assm}
 Throughout this section and in all expressions in
  Step 1 and 2 
  leading up to the
  calculation of $(\cB_\phi^\aleph)_\iota^{(k)}$ in \cref{cor:Bphi}, we suppose that
  %
  \bit
  \item $(Q^{(\iota)})_l\in\mathrm{FPC}_{\mathsf{M}}$;
  \item any expression of the form  $y=f(Q^{(\iota)})$ for a complex-valued
    function $f:\bbC\to\bbC$ will be shorthand for the composition of
    $\mathrm{rnd}_M \circ f$ where $\mathrm{rnd}_M :\bbC \to \mathrm{FPC}_M$
    is the rounding to the nearest element in $\mathrm{FPC}_M$ in absolute value.
   \eit
\label{assm:fp}
\end{assm}
In other words, we assume to be able to compute a floating-point expression for 
$(\cB_\phi^\aleph)_\iota^{(k)}\in\mathrm{FPC}_{\mathsf{M}}$ in \cref{cor:Bphi} starting with
$(Q^{(\iota)})_l\in\mathrm{FPC}_{\mathsf{M}}$ with {\it exact roundings}
throughout. On a computer implementation of Step 1 and 2, this can be carried
out by working if necessary with internal precision greater than $\mathsf{M}$ in the calculation
of functions $f(x)$ for $x\in\mathrm{FPC}_{\mathsf{M}}$, and then rounding-off-to nearest
when storing the value of $y=f(x)$ as an $\mathsf{M}$-bit complex floating point. We
have used the GNU MPC library \cite{mpc} (building on the GNU GMP/MPFR multi-precision
libraries \cite{gmp,mpfr}) to achieve this.

\subsubsection{Newton--Raphson inversion}
\label{sec:nr}

Let now $u^{(\iota)}$ be as in \eqref{eq:admQ2} and suppose we want to find an approximate
pre-image $Q^{(\iota)}\in\mathrm{FPC}_\mathsf{M}$ of $u^{(\iota)}=\chi_{j}(Q)$
with $\mathsf{M}$-bit precision. We can do this via the following adaptation of the
Newton--Raphson method.

\begin{algo}
Suppose let $Q_0=\exp(t(h_0))$ with $h_0$ given in
\eqref{eq:u0}, and consider the double sequence of linear systems
\beq
\sum_j \cJ_{ij}(Q^{(\iota,k-1,l/n)}) (Q^{(\iota,k,l/n)}-Q^{(\iota,k-1,l/n)}) =
\chi_{i}(Q^{(\iota,k-1,l/n)})-\frac{l u^{(\iota)}_i}{n}, \quad k,n,l\geq 1.
\label{eq:NR}
\eeq
For all $\iota$ set $Q^{(\iota,0,1)}=Q_0$ , $n=1$, $l=1$, and iterate the following:
\bit
\item assuming $\det \cJ_{ij}(Q^{(\iota,k-1,l/n)})\neq 0$,
solve \eqref{eq:NR} for $Q^{(\iota,k,n)}$ recursively in
  $k$;
\item if there exists $k_{l/n}\in \bbN$ such that
  $|Q^{(\iota,k_{l/n},l/n)}-Q^{(\iota,k_{l/n}-1,l/n)}| < \varepsilon_{Q^{(\iota,k_{l/n},l/n)}}$, we let
  $Q^{(\iota,0,(l+1)/n)}:=Q^{(\iota,k_{l/n},l/n)}$ and repeat the previous step. If there is no such $k_{l/n}$, let $n\to n+1$, $l \to \lfloor (n+1) l/n \rfloor$ and repeat the previous step.
  \eit
  \label{algo:NR}
\end{algo}
  
The adaptive re-scaling of the value of $u^{(\iota)}$ on the r.h.s. of
\eqref{eq:NR} ensures that the process converges to
$Q^{(\iota)}:=Q^{(\iota,k_1,1)}\in \mathrm{FPC}_{\mathsf{M}}$, such that the
$\varepsilon_{Q^{(\iota)}}$-ball $B_{\varepsilon_{Q^{(\iota)}}}(Q^{(\iota)})$ cointains at least one exact pre-image of $u_i^{(\iota)}=\chi_{i}(Q)$.\\

\subsubsection{Bounds and exact roundings}

In the following, for any $x\in\mathrm{FPC}_{\mathsf{M}}$ and
any symbol $\delta_x$, write $\hat \delta_x :=\delta_x+\varepsilon_x$. The first question we would like to ask is what kind of bounds can be put on the
error in the computation of $\tilde\DD_{i,j}(c,\iota)$ under \cref{assm:fp}
given the rounding error $\varepsilon_{Q^{(\iota)}}$ in the calculation of
$Q^{(\iota)}$.   For $c=0$, we could bound this a priori by $ \delta_{\tilde
  \DD_{i,j}(0,\iota)}(\varepsilon_{Q^{(\iota)}})$ defined as
\begin{align}
 \mathrm{sup}_{Q', Q'' \in B_{\varepsilon_{Q^{(\iota)}}}(Q^{(\iota)})}\Big|  \chi_{i}(Q'^j) - 
\chi_{i}(Q''^j) \Big|  \leq  \mathrm{sup}_{Q\in B_{\varepsilon_{Q^{(\iota)}}}(Q^{(\iota)})}\Big|
 \chi_{i}(Q^j)\Big|  - \mathrm{inf}_{Q\in B_{\varepsilon_{Q^{(\iota)}}}(Q^{(\iota)})}
\Big| \chi_{i}(Q^j) \Big| \nn \\   \leq  \sum_{w\in \Gamma(\omega_i)}  \bigg[ \prod_i \l(
\big|Q^{(\iota)}_i\big|^{j w_i} + (-1)^{\mathrm{sign} w_i} 
\epsilon^{j w_i}\r)  - \prod_i\l(
\big|Q^{(\iota)}_i\big|^{j w_i} - (-1)^{\mathrm{sign} w_i} 
\epsilon^{j w_i}\r)\bigg] \nn \\
 =:  \delta_{\tilde
  \DD_{i,j}(0,\iota)}(\varepsilon_{Q^{(\iota)}})-\varepsilon_{\tilde
  \DD_{i,j}(0,\iota)}(\varepsilon_{Q^{(\iota)}})
\label{eq:boundin}
\end{align}
where the second line is obtained by repeated use of the triangle
inequality, and the last line defines $ \delta_{\tilde \DD_{i,j}(0,\iota)}$ as
the sum of the theoretical error bound on the l.h.s. and a rounding error
term due to the floating-point round-off. With $c\neq 0$, we can use the fact that the real and imaginary
parts of $\chi_{i}(Q)$ are harmonic in $Q$ to obtain a uniform bound for the
the order-$|c|$ derivatives. Recall that if $f\in
C^\infty(\Omega\subset\bbR^n)$ is harmonic in $\Omega$, and $\Omega'\subsetneq
\Omega$ is compact, we have (see for example \cite[Thm~2.10]{MR0473443})
\beq
\mathrm{sup}_{\Omega'}|D^c f| \leq \l(\frac{n |c|}{\mathrm{dist}(\Omega',\de\Omega)}\r)^{|c|}\mathrm{sup}_\Omega |f|.
\label{eq:harm}
\eeq
Using \eqref{eq:harm} with $\Omega'=B_{\varepsilon_{Q^{(\iota)}}}(Q^{(\iota)})$,
$\Omega=B_{2{\varepsilon_{Q^{(\iota)}}}}(Q^{(\iota)})$, we get
\bea
\mathrm{sup}_{Q', Q'' \in B_{\varepsilon_{Q^{(\iota)}}}(Q^{(\iota)})}\Big|  D^c \chi_{i}(Q'^j) - 
D^c \chi_{i}(Q''^j) \Big|
& \leq & \l( 64 |c|\r)^{|c|} \delta_{\tilde
  \DD_{i,j}(0,\iota)}(2{\varepsilon_{Q^{(\iota)}}}) \nn \\ & =: & \delta_{\tilde \DD_{i,j}(c,\iota)}({\varepsilon_{Q^{(\iota)}}})-\varepsilon_{\tilde \DD_{i,j}(c,\iota)}({\varepsilon_{Q^{(\iota)}}})
\eea
where once again we included in the definition of the error bound
$\delta_{\tilde \DD_{i,j}(c,\iota)}$ a contribution accounting for the rounding
in $\mathrm{FPC}_{\mathsf{M}}$. From this we recursively get a bound on the error terms in an exactly-rounded FPC evaluation of the antisymmetric
characters of $V(\omega_j)$,
\bea
\delta_{\DD(c,\iota, \wedge^{k-m} V(\omega_j))}-\varepsilon_{\DD(c,\iota, \wedge^{k-m} V(\omega_j))} &:=&
    \sum_{m=1}^k 
  \sum_{c'+c''=c}  \l(\prod_l \binom{c_l}{c'_l}\r)
  \Big(\big|\tilde{\DD}_{m,j}(c',\iota)\big| \delta_{\DD_{m}(c'',\iota,
    \wedge^{k-m} 
    V(\omega_j))} \nn \\ &+& |\DD(c'',\iota, \wedge^{k-m} V(\omega_j))|
  \delta_{\tilde{\DD}_{m,j}(c',\iota)} + 2^{-\mathsf{M}}\Big)
\nn \\ & \geq & \mathrm{sup}_{Q', Q'' \in B_{\varepsilon_{Q^{(\iota)}}}(Q^{(\iota)})}\Big|
D^c\chi_{\wedge^k\mathfrak{e}_8} (Q') - 
  D^c \chi_{\wedge^k\mathfrak{e}_8}(Q'') \Big| 
\eea
where we assumed (and can verify a posteriori) that $\delta_{\tilde{\DD}_{m,j}(c',\iota)}\delta_{\DD(c'',\iota,
  \wedge^{k-m} V(\omega_j))} \leq 2^{-\mathsf{M}}$.

Now, by \cref{prop:FdB,prop:dJinv} and \eqref{prop:step1}, we have that
\beq
\mathfrak{D}^c \chi_{\wedge^k \mathfrak{e}_8}  =
\LL_{c,\iota}((D^d\chi_{\wedge^k \mathfrak{e}_8})_{|d|\leq|c|})
\label{eq:LL}
\eeq
for a linear functional $\LL_{c,\iota}$ independent on $k$: this linear operator
encodes the change-of-variable when writing the differential operator $\mathfrak{D}_c=\de^{|c|}_{\chi_{1}^{c_1}\dots
\chi_{8}^{c_8}}$ in exponentiated linear co-ordinates $Q$ on
$\cT_{E_8}$ (\cref{prop:FdB}), which gives a linear combination of $D^d$ at $Q^{(\iota)}$ with
$|d|\leq |c|$. So to figure out how the $\epsilon$-error in $Q^{(\iota)}$
propagates to the l.h.s. of \eqref{eq:LL}, we need to estimate a bound for the 2-norm of
$\LL_{c,\iota}$,
\bea
\delta_{\mathfrak{D}^c \chi_{\wedge^k \mathfrak{e}_8}} -\varepsilon_{\mathfrak{D}^c \chi_{\wedge^k \mathfrak{e}_8}} &:=& \|\LL_{c,\iota}\|_2
\| (\delta_{D^d\chi_{\wedge^k \mathfrak{e}_8}})_{|d|\leq |c|}\|_2
  \nn \\
  & \geq & \mathrm{sup}_{Q', Q'' \in B_{\varepsilon_{Q^{(\iota)}}}(Q^{(\iota)})} \l|
  \LL_{c,\iota} \l( D^d\chi_{\wedge^k \mathfrak{e}_8}(Q') - D^d\chi_{\wedge^k
    \mathfrak{e}_8}(Q'')\r) \r|.
  \label{eq:errchiwed}
\eea
This can be done as follows: from \cref{prop:dJinv}, we have
\beq
\|\cJ^{\rm inv}_{c,\iota}\|_{2} \leq \sum_{r=1}^{|c|} \sum_{d^{(1)}+\dots
  d^{(r)} = c} \| \cJ^{\rm
  inv}_{0,\iota}\|_2 \prod_{l=1}^r
\|\cJ_{d^{(r)},\iota}\|_2 ~\|\cJ^{\rm inv}_{0,\iota}\|_2
\label{eq:Jinvnorm}
\eeq
and we can place upper limits for the summands by bounding each operator
2-norm in $\mathrm{Mat}(8,\bbC)$ with the Frobenius norm, $\| \|_2 \to \| \|_F$, on which upper bounds can be
straightforwardly found in terms of $u^{(\iota)}_i$ by using \eqref{eq:harm}:
\bea
\|\cJ^{\rm inv}_{c,\iota}\|_{2} & \leq & \sum_{r=1}^{|c|} \sum_{d^{(1)}+\dots
  d^{(r)} = c} \| \cJ^{\rm
  inv}_{0,\iota}\|_F \prod_{l=1}^r
\|\cJ_{d^{(l)},\iota}\|_F ~\|\cJ^{\rm inv}_{0,\iota}\|_F,
\label{eq:Jinvnorm2}
\\
\|\cJ_{d,\iota}\|_F^2 & \leq & 8 {16 |d+1|}^{|d|+1} \l[ \sum_j \l(|u_j^{(\iota)}|+
  \delta_{\DD(0,\iota,V(\omega_j))}(2{\varepsilon_{Q^{(\iota)}}}) \r)^2 \r]\\
\| \cJ^{\rm inv}_{0,\iota}\|_F^2 & \leq & 128 \l[ \sum_j \l(|Q_j^{(\iota)}|+
  2{\varepsilon_{Q^{(\iota)}}} \r)^2 \r].
\label{eq:Jinvnorm3}
\eea
Suppose that $\| \cJ^{\rm inv}_{0,\iota}\|_F >1$. Then by \cref{prop:FdB} and
\eqref{eq:Jinvnorm2}--\eqref{eq:Jinvnorm3}, the largest contribution to $\|\LL_{c,\iota}\|_2$ comes
from the Fa\`a di Bruno graph with a single horizontal arrow emanating from
the top-left vertex, and all vertical arrows ending on it. Therefore,
\bea
\|\LL_{c,\iota}\|_2 & \leq & |c|! \sum_{r=1}^{|c|} 2^{\frac{17}{2} r}   \l[ \sum_j \l(|Q_j^{(\iota)}|+
  2{\varepsilon_{Q^{(\iota)}}} \r)^2 \r]^{r/2} \l[ \sum_j \l(|u_j^{(\iota)}|+
  \delta_{\DD(0,\iota,V(\omega_j))}(2{\varepsilon_{Q^{(\iota)}}}) \r)^2\r]^{r/2}\nn \\ & \times
& \sum_{d^{(1)}+\dots
  d^{(r)} = c}  \prod_{l=1}^r {16 \l|d^{(l)}+1\r|}^{(|d^{(l)}|+1)/2},
\label{eq:normLL}
\eea
placing accordingly a bound on $\delta_{\mathfrak{D}^c \chi_{\wedge^k
    \mathfrak{e}_8}}$ in \eqref{eq:errchiwed}.

The main question now is whether it is possible to choose judiciously the
sampling set $\mathfrak{S}\subset \bbQ+\ri \bbQ$ in \eqref{eq:admQ2} for
$\{u^{(\iota)}\}_{\iota\in\mathfrak{I}}$ such that we can employ the bounds
\eqref{eq:errchiwed}--\eqref{eq:normLL} to control the size of the error
$\delta_{\mathfrak{D}^c \chi_{\wedge^k
    \mathfrak{e}_8}}$, and hence rigorously round $\chi_{\wedge^k
  \mathfrak{e}_8}(Q^{(\iota)})$ to its exact value as a complex number
with rational real and imaginary parts. One na\"ive possibility is to choose
$\mathfrak{S}\subset \bbZ$, which would lead to $\chi_{\wedge^k
  \mathfrak{e}_8}(Q^{(\iota)})\in\bbZ $; however one immediately faces the
problem that that
even for small integer values of $u_i^{(\iota)}$, the values of the power sums
$\tilde \DD_{k,j}(c,\iota)$ may be as large as $\approx 10^{40}$ for large
$k$. This implies that the loss of significant digits coming from 
the round-off errors $\varepsilon_{\tilde \DD_{k,j}(c,\iota)}$ will be
correspondingly large, and lead eventually to a very poorly controlled bound on $\delta_{\mathfrak{D}^c \chi_{\wedge^k
    \mathfrak{e}_8}}$. The above indicates that to compute $(\cB_\phi^\aleph)$
exactly we need to 
\bit
\item pick values in $\mathfrak{S}$ that are sufficiently close to zero, so
  that the growth of the values of $\DD_{k,j}(c,\iota)$ and $\DD(c,\iota, \wedge^k \omega_j)$ is under control;
\item have sufficiently many sample points, with $|\mathfrak{S}|$ large enough
  to ensure that $\cA_\phi^\aleph$ is always non-singular\footnote{An
    immediate lower bound here is $|\iota^{\rm max}|=31$, corresponding to
    $\phi_j=\delta_{j7}$: this is the system computing coefficients of
    monomials of the form $\chi_{7}^k$, $k=1,\dots, |\iota^{\rm max}|=31$.}; 
\item find an a priori reliable prescription to round the final values for $\mathfrak{D}^c \chi_{\wedge^k
    \mathfrak{e}_8}$ to a number in $\bbQ+\ri \bbQ$.
\eit
A particularly convenient choice is to take
\beq
\mathfrak{S}:=\l\{ \frac{l}{2}+\ri\frac{m}{2},l\in\{-3,\dots,3\};
m\in\{-2,\dots,2\}\r\},
\label{eq:smpl}
\eeq
so that $|\mathfrak{S}|=40$. Then for $u_j^{(\iota)} \in \mathfrak{S}$, we would
expect to have
\beq
\mathfrak{D}^c \chi_{\wedge^k \mathfrak{e}_8 }(Q^{(\iota)})\in 2^{-|\iota^{\rm max}|} (\bbZ+\ri \bbZ)
\eeq
from \eqref{eq:e8num} and \eqref{eq:iotamax}. The bounds computed in the
previous section can be combined together to prove the following

\begin{thm}
For any $(\phi,\aleph)$, there exists a choice of values $\{u_j^{(\iota)}\}_{\iota\in\mathfrak{I}}$ in
\eqref{eq:admQ2}, with
\beq
\mathfrak{S}=\l\{\bary{cc} \{-e_k,\dots,e_k-1,e_k\}, & \phi_k=|\phi|=1, \\ \l\{
\frac{l}{2}+\ri\frac{m}{2}\r\}_{\stackrel{l=-3\dots 3}{m=-2,\dots, 2}}, & |\phi|>1 \eary \r.
\eeq
and of approximate pre-images
$\{Q_j^{(\iota)}\in\mathrm{FPC}_{\mathsf{M}}\}_{\iota\in\mathfrak{I}}$ satisfying the following:
\bit
\item in \eqref{eq:Aphi}, $\det\mathfrak{Re}\cA_\phi^{\aleph} \neq 0$ for all
  $(\phi, \aleph)$;
\item $0\neq |\det \cJ(Q^\iota)|\in \mathrm{FPC}_\mathsf{M}$, $\|\cJ^{\rm inv}\|_F$ in \eqref{eq:jac};
\item for $\mathsf{M}\geq \mathsf{M}_0:= 494$, we have that
  \beq
  2^{|\iota^{\rm max}|} \delta_{\mathfrak{D}^c \chi_{\wedge^k \mathfrak{e}_8
  }}< \frac{1}{2}.
  \eeq
\eit 
\label{thm:round}
\end{thm}
\begin{proof}
We can prove the first two points by exhibiting a choice of
  values $\{(u,Q)^{(\iota)}\}_{\iota\in\mathfrak{I}}$, where $Q^{(\iota)}$ is
  found with \cref{algo:NR}; we claim that one such can be found, and its explicit form for all $\iota$ is available upon request. The third point follows from
  evaluating the analytic bounds \eqref{eq:boundin}--\eqref{eq:normLL} for the
    set of pre-images $Q^{(\iota)}$ thus found; we refer the reader to the
    ancillary files described in \cref{sec:auxfiles} for the calculation of
    all the relevant quantities.
\end{proof}

Let us pause to rephrase the content of \cref{thm:round} once more. The theorem states that there exists a choice of $u^{(\iota)}$ compatible with
\eqref{eq:admQ2} and \eqref{eq:smpl} such that $\mathfrak{Re}\cA_\phi^\aleph$ is
a non-singular linear system over $\bbZ$ for all $(\phi, \aleph)$; and
furthermore, there exists a sufficiently large integer $\mathsf{M}_0$ such
that, 
denoting  $Q^{(\iota)}\in\bbC$ and $\widehat Q^{(\iota)}\in\mathrm{FPC}_\mathsf{M_0}$, respectively,
a pre-image of $u^{(\iota)}$ in \eqref{eq:admQ2} and an $\mathsf{M}$-bit
precision approximation thereof, the value of $\mathfrak{D}^c \chi_{\wedge^k
  \mathfrak{e}_8 }(\widehat Q^{(\iota)})\in\mathrm{FPC}_{\mathsf{M_0}}$ computed under
\cref{assm:fp} with $\mathsf{M}\geq \mathsf{M}_0$ must satisfy
\beq
2^{|\iota^{\rm max}|} \mathfrak{D}^c \chi_{\wedge^k \mathfrak{e}_8
}(Q^{(\iota)}) \in B_{\frac{1}{2}}(\mathfrak{D}^c \chi_{\wedge^k \mathfrak{e}_8
}(\widehat Q^{(\iota)})),
\eeq
by the analytic bounds \eqref{eq:boundin}--\eqref{eq:normLL}. Now,
the l.h.s. is a point in the two-dimensional lattice $\bbZ^2
\simeq \bbZ+\ri \bbZ$; and the Theorem constrains it to lie in a disk of
radius $<1/2$ centred at $\mathfrak{D}^c \chi_{\wedge^k \mathfrak{e}_8}
(\widehat Q^{(\iota)})$: so it's the {\it unique} (if at all existent) integer lying in
that disk. Therefore, the exact value of $\mathfrak{D}^c \chi_{\wedge^k \mathfrak{e}_8
}(Q^{(\iota)})$ is determined by $\mathfrak{D}^c \chi_{\wedge^k \mathfrak{e}_8}
(\widehat Q^{(\iota)})$
 computed from Step 1 and 2 under \cref{assm:fp}. This concludes the
calculation of $\mathfrak{D}^c \chi_{\wedge^k
  \mathfrak{e}_8}|_{\chi_{j}=u_j^{(\iota)}}$, and thus
$(\cB_\phi^\aleph)_\iota^{(k)}$, from \eqref{eq:admQ2}.

 \subsection{Solving exactly for $N_\iota^{(k)}$.}
\label{sec:solvext}
 
 \cref{thm:round} allows us to write the exact form of
 $\mathfrak{L}_\phi^\aleph$ subordinate to a choice of sampling set
 \eqref{eq:admQ2} and \eqref{eq:smpl}; all is left to do to provide the
 solution of \cref{prob:e8num} is to solve explicitly for $N_\iota^{(k)}$
the linear problem $\mathfrak{L}_\phi^\aleph$ for all $\phi$ and $\aleph$. By
 \cref{thm:round}, and since $N_\iota^{(k)}\in \bbZ$, we just need to consider
 the real part of $\mathfrak{L}_\phi^\aleph$, which is a linear system over
 $\bbQ$. This can be solved exactly using
 {\bf Method 4}
by Dixon's $p$-adic lifting \cite{MR681819}: we fix
a prime $p$, invert $\mathfrak{Re} \cA_\phi^\aleph$ mod $p$ by LU
decomposition, and solve for $N_\iota^{(k)}$ mod $p^Q$ (for $Q$ large
enough). This allows to reconstruct $N_\iota^{(k)}$ from its $p-$adic
expansion in $\cO(r^3\log^2 r)$ time, where
$r=|\mathfrak{I}_\phi^\aleph|$. For this computation, the \texttt{fmpz\_mat} module
of the FLINT (Fast Library for Number Theory) C~library has been
systematically employed; see \cite{flint} and references therein for details.

\subsection{Implementation: runtime estimates and distributed computation}
\label{sec:impl}

An estimate of the runtime required can be illustrated by the
following two tables. For the first six rows, the operations in \cref{tab:runt} are in
$\mathsf{M}+1=495$-bit precision with correct rounding-to-nearest, whereas in
the last row multi-precision rational arithmetic is used. The first five rows are for a {\it fixed}
$\iota\in\mathfrak{I}_{\phi}^\aleph$ and {\it all}
$c\in\mathscr{C}_{\gamma_\aleph}$; the last two indicate estimates for
all $\iota\in\mathfrak{I}_{\phi}^\aleph$ for given $(\phi,\aleph)$.

\begin{table}[!h]
  \begin{tabular}{|c|c|}
    \hline
    Operation  & Walltime estimate \\
        \hline
        $Q=Q(u^{(\iota)})$ (\cref{algo:NR}) & $\approx 10$ sec \\
        \hline
        $\DD(c,\iota, \wedge^k \mathfrak{e}_8)$,
          \eqref{eq:derQpowk}-\eqref{eq:derQwedk} 
&        $\approx \binom{|c|+8}{8}~5~\mathrm{sec}$ \\
        \hline
        $\DD(c,\iota, V(\omega_i)$, \eqref{eq:derQpowk}--\eqref{eq:Vom2}
&        $\approx \binom{|c|+8}{8}~1~\mathrm{sec}$ \\
        \hline
        $\cJ^{\rm inv}_{c,\iota}$ \eqref{eq:Jinviota}
        &        $\approx \binom{|c|+8}{8}~.1~\mathrm{sec}$ \\
        \hline
        $\mathfrak{D}_c$ (\cref{def:FdB}, \eqref{eq:FdB})
        &        $\approx \binom{|c|+8}{8}~.1~\mathrm{sec}$ \\
%
%
        \hline
        $\{\mathfrak{D}_c \chi_{\wedge^k
          \mathfrak{e}_8}(u^{(\iota)})\}_{\iota\in\mathfrak{I}_\phi^\aleph}$ &
        $\approx \binom{|c|+8}{8} |\mathfrak{I}_\phi^\aleph| \times 5\times 10^{-2}
        \mathrm{sec}$ \\
        \hline
        $N_\iota^{(k)}$ &
        $\approx |\mathfrak{I}_\phi^\aleph|^3 \log^2 |\mathfrak{I}_\phi^\aleph|
        \times 10^{-5} \mathrm{sec}$\\
        \hline
\end{tabular}
  \caption{Wall-clock estimate of the sequence of operations in \cref{sec:complphi,sec:compbphi}.} 
  \label{tab:runt}
\end{table}

Because $|\mathfrak{I}|\approx 10^6$, and $\binom{d_{\rm max}+8}{8}\approx 10^3$, the total runtime would be in the order of a few dozen years, but there are a
number of ways to reduce it considerably. Define
\beq
(\zeta_{\aleph,\phi})_i= \left\{\bary{cc} \delta_{i\sigma(\phi,1)}+\delta_{i\sigma(\phi,2)}, &  \aleph=(m_1,m_2),\\
  \delta_{i\sigma(\phi,1)} , & \aleph=m_1, \\
  0, & \aleph>d_{\rm max}. \eary \right.
\eeq
and consider monomial subsets $\mathfrak{I}_\phi^\aleph$,
$\mathfrak{I}_{\phi'}^{\aleph '}$ such that
\beq
\phi_i +(\zeta_{\aleph,\phi})_i = \phi'_i +(\zeta_{\aleph',\phi'})_i \quad
\forall i.
\eeq
Then the condition \eqref{eq:admQ2} is the same for both subsets: in
particular, if $|\mathfrak{I}_\phi^\aleph|\leq
|\mathfrak{I}_{\phi'}^{\aleph'}|$, admissible sampling values $u^{(\iota)}$
for $\iota\in\mathfrak{I}_\phi^\aleph$ are also admissible as sampling values
for $\mathfrak{I}_{\phi'}^{\aleph'}$, since \eqref{eq:admQ2} is the same: the
calculation of the first five rows can then be performed only once. The
repetition allows to drastically reduce the number of values $Q=Q(u^{(\iota)})$ on which
the first five operations of \cref{tab:runt} have to be performed by a factor
of about five, down to $\approx 2\times 10^5$ points, and a total runtime for
these operation of $\approx 7.5 \mathrm{yrs}$.

Furthermore, computations at different sampling values are entirely independent of each
other, which means that they can be effectively parallelised purely by
segmentation of the sampling set. Since these processes do not interact with
each other, distributing the first batch of five calculations in
\cref{tab:runt} to $\mathsf{N}$ CPU cores gives a theoretical net factor of
$\mathsf{N}$ in the reduction of the absolute walltime for the
computation: with $\mathsf{N}\approx 100$, corresponding to two
4-CPU machines with 12 processor cores each, this is reduced down to
$\approx 4~\mathrm{weeks}$. Similar considerations apply for the last two sets of
computations in \cref{tab:runt}: $\mathfrak{L}_\phi^\aleph$ and
$\mathfrak{L}_{\phi'}^{\aleph'}$ can be computed and solved in parallel whenever
$\mathfrak{I}_\phi^\aleph$ and $\mathfrak{I}_{\phi '}^{\aleph '}$ are
incomparable under the weak order $<$ of \eqref{eq:preord1}, for an extra
$\approx 1.5~\mathrm{weeks}$ worth of elapsed real time for the computation of
$N_\iota^{(k)}$.

\subsection{Results and ancillary files} The calculations described in the
previous section were carried out at the Maths Compute Cluster at Imperial
College London, and to a smaller extent, at the {\tt{Omega}} and
{\tt{Muse}} compute clusters of the Universit\'e de Montpellier. The results of the computation, including the coefficients $N_\iota^{(k)}$  as well as the
auxiliary quantities $Q^{(\iota)}$, $u^{(\iota)}$ and $\delta_\iota^{\rm max}:=\mathrm{sup}_{k,c} 2^{|\iota^{\rm max}|} \delta_{\mathfrak{D}^c \chi_{\wedge^k \mathfrak{e}_8
  }}$ relevant for \cref{thm:round} are stored in binary
files available at 
{\begin{center} \texttt{http://tiny.cc/E8Char},\end{center}}

which also comprises of a {\it
  Mathematica} notebook to access them. Their structure is
described in \cref{sec:auxfiles}.
the
polynomial character decomposition $\mathfrak{p}_k$ of $\chi_{\wedge^k
  \mathfrak{e}_8}$ is reproduced in \cref{sec:bigtable} up to $k=11$.

\section{Applications}
\label{sec:appl}

A big chunk of the ``Applications'' section of this paper is really \cite{Brini:2017gfi}; the
polynomial character decompositions found here were used there to construct
the spectral curves of the $\widehat{E_8}$-relativistic Toda chain
(\cref{sec:is}) and  the Seiberg--Witten curves of $E_8$ minimallly
supersymmetric Yang--Mills theory on $\bbR^{1,3}\times S^1$ (\cref{sec:sw}),
as well as to prove an all-genus version of the Gopakumar--Ooguri--Vafa
correspondence (\cref{sec:cs}) for the Poincar\'e sphere, and to provide
a mirror theorem for the Fano orbicurve of type $E_8$
(\cref{sec:polfrob}). We provide here some examples and non-examples of
characteristic polynomials having Galois
group the Weyl group of $E_8$ (\cref{sec:galois}).

\subsection{Some cyclotomic cases}

Let us first start with some non-examples, where the splitting field of the
characteristic polynomial \eqref{eq:cpol} is given by a finite sequence of
cyclotomic extensions of the rationals.

\begin{example}
The obvious one -- a sanity check
really -- is given by specialising
\beq
\chi_{i}(g)=\dim V(\omega_i)
\label{eq:dimvi}
\eeq
corresponding to $g=1$. Then obviously we must have
\beq
\chi_{\wedge^k \mathfrak{e}_8}(g) = \binom{248}{k}
\eeq
which is (non-trivially) verified by substituting \eqref{eq:dimvi} into
\eqref{eq:e8num}, leading to to
\beq
\mathfrak{Q}_{240}(\mu)= (\mu-1)^{240}.
\eeq
\end{example}
\begin{example}
  \label{ex:chi0}
A more interesting case is
\beq
\chi_{i}(g)=0.
\label{eq:chi0}
\eeq
Our solution of \cref{prob:e8num} then leads to the following result for the
antisymmetric adjoint traces,
\beq
\chi_{\wedge^k \mathfrak{e}_8}(g) = \l\{\bary{cc} 0 & 31 \not | ~k,
\\ \binom{8}{k/31} & \mathrm{else.}  \eary \r.
\eeq
Therefore,
\beq
\mathfrak{Q}_{240}(\mu) = \l(\Phi_{31}(\mu) \r)^8=\l(\sum_{i=0}^{30} \mu^i \r)^8,
\eeq
where $\Phi_{n}$ denotes the $n^{\rm th}$ cyclotomic polynomial,
\beq
\Phi_n(\mu)=\prod_{\stackrel{1\leq k\leq
    n}{\mathrm{gcd}(k,n)=1}}\l(x-\exp\frac{2\pi\ri k}{n} \r),
\eeq
so in this case the splitting field of $\mathfrak{Q}_{240}$ is
$\bbQ(\exp(2\pi \ri/31))$, and the Galois group is $\bbZ/(31\bbZ)$.
\end{example}

\begin{example}
A third
notable example, already considered in \cite{Brini:2017gfi}, is
\beq
(\chi_{i})_i(g)= (1, 3, 0, 3, -3, 3, -2, -2).
\eeq
Geometrically, this corresponds to the super-singular limit of the Toda
spectral curves computed in \cite{Brini:2017gfi}; equivalently, the
superconformal point of the geometrically engineered $E_8$ Yang--Mills theory
on $\bbR^{1,4}$; the zero 't Hooft coupling point of Chern--Simons theory on
the Poincar\'e sphere; and the conifold point for the orbifold Gromov--Witten
theory of $[\cO(-1)^{\oplus 2}_{\bbP^1}/\mathbb{I}_{120}]$. We obtain in this case:
\bea
\mathfrak{Q}_{240} &=& \Phi_2^2(\mu) \Phi_6^3(\mu) \Phi_3^5(\mu) \Phi_5^5(\mu) \Phi_{18}^4(\mu)
\Phi_9^3(\mu)  \Phi_{15}^2(\mu) \Phi_{30}^3(\mu) \Phi_{45}^2(\mu)\Phi_{90}^3(\mu)\nn \\ &=
& (\mu +1)^2 \left(\mu ^2-\mu +1\right)^3 \left(\mu ^2+\mu +1\right)^5
\left(\mu ^4+\mu ^3+\mu ^2+\mu +1\right)^5 \left(\mu ^6-\mu ^3+1\right)^4
\nn \\ & & \left(\mu ^6+\mu ^3+1\right)^3 \left(\mu ^8-\mu ^7+\mu
   ^5-\mu ^4+\mu ^3-\mu +1\right)^2 \left(\mu ^8+\mu ^7-\mu ^5-\mu ^4-\mu
^3+\mu +1\right)^3 \nn \\ & & \left(\mu ^{24}-\mu ^{21}+\mu ^{15}-\mu ^{12}+\mu ^9-\mu ^3+1\right)^2 \left(\mu ^{24}+\mu ^{21}-\mu
   ^{15}-\mu ^{12}-\mu ^9+\mu ^3+1\right)^3.
\eea
\end{example}

\begin{example}

As a final example, consider
\beq
(\chi_{i})_i(g)= (0, -1, 0, 0, 1, 0, 0, 0).
\label{eq:exfin}
\eeq
This example was found heuristically while constructing a set of sampling
values $Q^{(\iota)}$ satisfying the three conditions of \cref{thm:round}. This
was done by generating sampling values with a uniform random measure and testing
{\it a posteriori} that they were admissible under the conditions of
\cref{thm:round}; only in one case did we find a set of values in
$\mathfrak{S}$ for which $\det\cJ(Q^{(\iota)})=0$, which is precisely
\eqref{eq:exfin}.
The singularity of the Jacobian is a sufficient condition for the Galois
group of the splitting field of $\mathfrak{Q}_{240}$ to fail to be the full
Weyl group -- this indicates that the corresponding group element is on the
boundary of a Weyl chamber, and hence the map from the Weyl group to the
Galois group has a kernel containing the corresponding reflection. This is however not necessary --  \cref{ex:chi0} is a
counter-example to this.

Switching back to calculations, the solution of \cref{prob:e8num} gives
\beq
\chi_{\wedge^k \mathfrak{e}_8}(g) = \l\{\bary{cc}  (-1)^{k+j} \tau_{\lfloor \frac{j}{4}\rfloor} & k = \lfloor \frac{7j}{2}
\rfloor , ~j \in \bbN
\\ 0 & \mathrm{else,}  \eary \r.
\eeq
where $\{\tau_j\}_{j=0}^{9}=\{1,13,82,334,985,2233,4030,5914,7144\}$. This gives
  \bea
  \mathfrak{Q}_{240}(\mu)  &=& \Phi_1^{2}(\mu) \Phi_2^{10}(\mu)\Phi_3(\mu)
  \Phi_{7}^{9}(\mu)\Phi_{12}^4(\mu) \Phi_{14}^{10}(\mu)\Phi_{84}^4(\mu)\nn \\
&=&  (\mu-1)^2 (\mu+1)^{10}\left(\mu ^2+\mu +1\right) \left(\mu ^4-\mu ^2+1\right)^4
  \left(\mu ^6-\mu ^5+\mu ^4-\mu ^3+\mu ^2-\mu +1\right)^{10} \nn \\ & & \left(\mu ^6+\mu ^5+\mu ^4+\mu ^3+\mu ^2+\mu
   +1\right)^9 \left(\mu ^{24}+\mu ^{22}-\mu ^{18}-\mu ^{16}+\mu ^{12}-\mu
   ^8-\mu ^6+\mu ^2+1\right)^4.\nn \\
   \eea
  \end{example}

\subsection{The Jouve--Kowalski--Zywina polynomial}

Let us move on to consider the irreducible examples for which the splitting
field has Galois group the full Weyl group $\cW(E_8)$. In \cite{MR2523316}, 
Jouve, Kowalski and Zywina give an example of such an explicit integral
polynomial in their Appendix, precisely by considering the characteristic
polynomial of a specific group element $g\in E_8$: hence this should be
reproducibile from the general knowledge of $N_\iota^{(k)}$ for a specific
choice of $\chi_{i}$. I will just content myself to reproduce their result from the
solution of \cref{prob:e8num} without proof: set
\beq
(\chi_{i})_i(g)= (38412, 699221720, 8046927290936, 42593483592, 175914484, 531468, 1044,
  4538992).
\label{eq:chijkz}
\eeq
I claim that these correspond to the values for the regular fundamental
characters of the special group element $g$ chosen in \cite{MR2523316}; this can
be checked directly from the definition of $g$ in
\cite[Section~3]{MR2523316}. Using the solution found for $N_\iota^{(k)}$, we
get
\bea
(\chi_{\wedge^k \mathfrak{e}_8}(g))_{k=1}^{120} &=&
(1044,532512,177003376,43147804716,8230609109252,1280164588118952, \nn \\ & &
167041280674255148,18671692028344452040,1816777039210236799436, \dots
\nn \\ & &
\dots, 1.0653573147903152502124530638226214941071412605209\times 10^{108})
\eea
from which we get
\beq
\mathfrak{Q}_{240}(\mu)=\mu^{120} \sum_{k=0}^{120} q_k (\mu+\mu^{-1})^k
\eeq
with
\bea
q_k &=&
(-1036,524076,-172657460,41688975082,-7871527038772,1211012431626440, \nn \\ &
& -156184748605164508,
17242140511966984109, -1655565532193307303324,
\nn \\ & & \dots, 3.6558789498396792285456042110665879483526916202580\dots\times
10^{86})
\eea
where $q_k$ match all coefficients in the final table of \cite[Appendix~B]{MR2523316}.

\subsection{The group theory lift of the Shioda polynomial}

Another example of unreduced Galois group is provided by Shioda in
\cite[Prop.~2.2]{MR1123543} by exhibiting an explicit even integral polynomial
$\Psi_{240}(v)$ (see {\it ibidem}, Eq.~(2.3)) satisfying the sufficient
conditions of \cite[Lemma 3.2]{MR2523316} for the Galois group to be the full
Weyl group of $E_8$. It turns out that his example is a Lie-algebraic limit of the
characteristic polynomials constructed here: it can be verified directly
using Newton--Girard identities and the explicit form of the root module
for $\mathfrak{e}_8$ that there exists a regular element $h\in\mathfrak{e}_8$
such that
\beq
\Psi(v)=v^{-8}\det_{\mathfrak{e}_8}(v-h)=\prod_{\alpha\in\Delta}(v-\a \cdot h)
\label{eq:psish}
\eeq
from which we can read off directly the value of the regular fundamental
characters $\chi_{i}(\re^h)$. We obtain from \cite[Eq.~(2.3)]{MR2561902} that
\beq
\chi_{i}(\re^h)=(3811, 6967872, 6529441632, 128894934, 1829184, 23404,184, 129266)_i.
\eeq
We can use the solution of \cref{prob:e8num} to construct the characteristic
polynomial for the group theory lift of \eqref{eq:psish}, which by the same
token has Galois group the full Weyl group. We get
\beq
\mathfrak{Q}_{240}(\mu)=\mu^{120} \sum_{k=0}^{120} \tilde q_k (\mu^k+\mu^{-k})
\eeq
with
\bea
\tilde q_k &=&
(-1,-176,-22152,-1680656,-119102436,-5862463536,-264263737336,\nn
\\ & & -13802737369104,
-447652929177642,-3448304727308240,469573988622035048,\nn
\\ & & 22268833325609862288,
 474582729791826806540,1387986435785748203824,
\nn \\ & & \dots, 8.68482933306707881880061436471\dots\times10^{87}).
\eea

\begin{appendix}

  \section{Auxiliary files}
  \label{sec:auxfiles}

  The \textsf{E8Char} package available at \texttt{http://tiny.cc/E8Char}
  consists of five binary files containing numerical data for the set of
  admissible exponents $\iota$, the coefficients $N_\iota^{(k)}$ of the
  solution of \cref{prob:e8num},
sampling values, pre-images and error bounds  $u^\iota$, $Q^\iota$ and $\delta_\iota^{\rm max}$, plus a {\it Mathematica} notebook and further
  weight data in ASCII format for $\Gamma(\omega_i)$, $i=1,7,8$. The structure
  of the binary files is described below.

  \begin{longtable}{|l|p{13cm}|}
    \hline
    Filename & Description \\
    \hline
    {\tt DictFile.bin} & This stores the admissible exponents $\iota$ as an
    $8$-bit integer
    array $\{\{\iota_j\}_{j=1}^8\}_{\iota\in\mathfrak{I}}$. \\
    \hline
    {\tt SolFileE8.bin} & This stores the solution coefficients
    $N_\iota^{(k)}$ as a $64$-bit integer
    array $\{\{N_\iota^{(k)}\}_{k=1}^{120}\}_{\iota\in\mathfrak{I}}$. \\
    \hline
    {\tt USamples.bin} & This stores the sampling values
    $u^\iota=:(p^\iota/q^{\iota}+\ri r^\iota/s^\iota)$ as an $8$-bit integer
    array $\{\{p^\iota_j, q^\iota_j, r^\iota_j,
    s^\iota_j\}_{j=1}^8\}_{\iota\in\mathfrak{I}}$. \\
    \hline
    {\tt AllRoots.bin} & This stores the $496$-bit precision pre-images
    $Q^\iota$ of the
    sampling values  $u^\iota$ as a $512$-bit floating point
    array $\{\{\mathfrak{Re} Q^\iota_j, \mathfrak{Im}
    Q^\iota_j\}_{\iota\in\mathfrak{I}}$. Each $512$-bit block is structured as
    follows: the first 16-bit limb represents the exponent as a signed 16-bit
    integer; the next bit represents the sign; the remaining 495 bits store
    the mantissa. \\
    \hline
        {\tt Delta.bin} & This stores the maximum values on the upper
        bounds $\delta_\iota^{\rm max}$ as a $64$-bit floating point     array $\{\delta_\iota\}_{\iota\in\mathfrak{I}}$. \\
    \hline
    \end{longtable}

  The C~source code for the calculations in
  \cref{sec:complphi,sec:compbphi,sec:impl} leading up to the results of the
  above files is available upon request.

  
  \section{Table of exterior character relations}
  \label{sec:bigtable}

  The table below gives the polynomial character decompositions of
  $\chi_{\wedge^k \mathfrak{e}_8}$ for $k$ up to 11. 
  
  \begin{longtable}{|c|p{15cm}|}
    \hline
    $k$ & $\chi_{\wedge^k \mathfrak{e}_8}=\mathfrak{p}_k(\chi_{1}, \dots, \chi_{8})$ \\ \hline  
 1 & $\chi_{7}$\\ \hline
 2 & $\chi_{6}$ + $\chi_{7}$\\ \hline
 3 & $\chi_{5}$ $-$ $\chi_{7}$ + $\chi_{7}^2$\\ \hline
 4 & $\chi_{4}$ $-$ $\chi_{6}$ + $\chi_{6}$
 $\chi_{7}$\\ \hline
 5 & $\chi_{3}$ $-$ $\chi_{7}$ + $\chi_{5} \chi_{7}$ $-$ 2 $\chi_{6} \chi_{7}$ $-$ $\chi_{7}^2$ + 
    $\chi_{7}^3$\\ \hline
 6 & $-$$\chi_{1}^2$ $-$ $\chi_{1}^3$ + 2 $\chi_{1} \chi_{2}$ $-$ $\chi_{3}$ $-$ $\chi_{1} \chi_{4}$ + 
    $\chi_{1} \chi_{5}$ $-$ $\chi_{1} \chi_{6}$ $-$ $\chi_{1}^2$ $\chi_{6}$ + $\chi_{2} \chi_{6}$ + 
    $\chi_{5} \chi_{6}$ $-$ $\chi_{6}^2$ $-$ 2 $\chi_{1} \chi_{7}$ + $\chi_{1}^2$ $\chi_{7}$ + 
    $\chi_{2} \chi_{7}$ + $\chi_{4} \chi_{7}$ $-$ 2 $\chi_{6} \chi_{7}$ + $\chi_{1} \chi_{6} \chi_{7}$ $-$ 
    $\chi_{7}^2$ + 2 $\chi_{1} \chi_{7}^2$ + $\chi_{6} \chi_{7}^2$ $-$ $\chi_{1} \chi_{8}$ + 
    $\chi_{2} \chi_{8}$ $-$ $\chi_{6} \chi_{8}$
 $-$ $\chi_{7} \chi_{8}$\\ \hline
 7 & $\chi_{1}$ + 2 $\chi_{1}^2$ + $\chi_{1}^3$ $-$ $\chi_{1} \chi_{2}$ + $\chi_{2}^2$ $-$ $\chi_{3}$ $-$ 
    2 $\chi_{1} \chi_{3}$ + $\chi_{4}$ + $\chi_{1} \chi_{4}$ $-$ $\chi_{5}$ $-$ 3 $\chi_{1} \chi_{5}$ $-$ 
    2 $\chi_{1}^2$ $\chi_{5}$ + 2 $\chi_{2} \chi_{5}$ + $\chi_{5}^2$ + $\chi_{6}$ + 
    3 $\chi_{1} \chi_{6}$ + 2 $\chi_{1}^2$ $\chi_{6}$ $-$ $\chi_{2} \chi_{6}$ + $\chi_{4} \chi_{6}$ $-$ 
    2 $\chi_{5} \chi_{6}$ + $\chi_{6}^2$ + $\chi_{1} \chi_{6}^2$ + 2 $\chi_{7}$ + 
    4 $\chi_{1} \chi_{7}$ $-$ $\chi_{1}^2$ $\chi_{7}$ $-$ 2 $\chi_{1}^3$ $\chi_{7}$ $-$ 2 $\chi_{2} \chi_{7}$ + 
    4 $\chi_{1} \chi_{2} \chi_{7}$ $-$ $\chi_{3} \chi_{7}$ + $\chi_{4} \chi_{7}$ $-$ 4 $\chi_{5} \chi_{7}$ + 
    2 $\chi_{1} \chi_{5} \chi_{7}$ + 4 $\chi_{6} \chi_{7}$ + $\chi_{1} \chi_{6} \chi_{7}$ + 
    $\chi_{7}^2$ $-$ 6 $\chi_{1} \chi_{7}^2$ + $\chi_{1}^2$ $\chi_{7}^2$ + 2 $\chi_{2} \chi_{7}^2$ + 
    2 $\chi_{5} \chi_{7}^2$ $-$ 2 $\chi_{6} \chi_{7}^2$ $-$ 4 $\chi_{7}^3$ + 2 $\chi_{1} \chi_{7}^3$ + 
    $\chi_{7}^4$ $-$ $\chi_{1} \chi_{8}$ $-$ $\chi_{1}^2$ $\chi_{8}$ + $\chi_{2} \chi_{8}$ $-$ 
    $\chi_{4} \chi_{8}$ $-$ $\chi_{1} \chi_{6} \chi_{8}$ $-$ $\chi_{7} \chi_{8}$ $-$ 
    $\chi_{1} \chi_{7} \chi_{8}$ $-$ $\chi_{6}$
    $\chi_{7} \chi_{8}$ + $\chi_{1}$
    $\chi_{8}^2$\\ \hline
 8 & $-\chi_{1}^4$ + $\chi_{2}$ + 3 $\chi_{1}^2$ $\chi_{2}$ $-$ $\chi_{2}^2$ + $\chi_{3}$ $-$ 
    3 $\chi_{1} \chi_{3}$ $-$ $\chi_{4}$ + 2 $\chi_{1} \chi_{4}$ $-$ $\chi_{1}^2$ $\chi_{4}$ + 
    $\chi_{2} \chi_{4}$ + 2 $\chi_{5}$ $-$ $\chi_{1} \chi_{5}$ + $\chi_{1}^2$ $\chi_{5}$ $-$ 
    2 $\chi_{2} \chi_{5}$ + 2 $\chi_{4} \chi_{5}$ $-$ 2 $\chi_{5}^2$ + 2 $\chi_{1} \chi_{6}$ + 
    $\chi_{1}^2$ $\chi_{6}$ $-$ $\chi_{1}^3$ $\chi_{6}$ + 2 $\chi_{1} \chi_{2} \chi_{6}$ $-$ 
    2 $\chi_{3} \chi_{6}$ + $\chi_{4} \chi_{6}$ $-$ $\chi_{5} \chi_{6}$ + $\chi_{1} \chi_{5} \chi_{6}$ + 
    $\chi_{6}^2$ + $\chi_{1} \chi_{6}^2$ $-$ 2 $\chi_{7}$ + 5 $\chi_{1} \chi_{7}$ $-$ 
    $\chi_{1}^2$ $\chi_{7}$ + $\chi_{2} \chi_{7}$ + 2 $\chi_{1} \chi_{2} \chi_{7}$ $-$ 
    3 $\chi_{3} \chi_{7}$ + $\chi_{1} \chi_{4} \chi_{7}$ + $\chi_{5} \chi_{7}$ $-$ 
    3 $\chi_{1} \chi_{5} \chi_{7}$ + $\chi_{6} \chi_{7}$ + 2 $\chi_{1} \chi_{6} \chi_{7}$ + 
    2 $\chi_{2} \chi_{6} \chi_{7}$ + $\chi_{5} \chi_{6} \chi_{7}$ + $\chi_{6}^2$ $\chi_{7}$ + 
    3 $\chi_{7}^2$ $-$ $\chi_{1} \chi_{7}^2$ + 2 $\chi_{1}^2$ $\chi_{7}^2$ $-$ $\chi_{2} \chi_{7}^2$ + 
    2 $\chi_{4} \chi_{7}^2$ $-$ 4 $\chi_{5} \chi_{7}^2$ + 3 $\chi_{1} \chi_{6} \chi_{7}^2$ + 
    $\chi_{7}^3$ $-$ 2 $\chi_{1} \chi_{7}^3$ + $\chi_{6} \chi_{7}^3$ $-$ 2 $\chi_{7}^4$ + $\chi_{8}$ + 
    $\chi_{1} \chi_{8}$ $-$ $\chi_{2} \chi_{8}$ + $\chi_{1} \chi_{2} \chi_{8}$ $-$ 3 $\chi_{3} \chi_{8}$ + 
    2 $\chi_{4} \chi_{8}$ $-$ 3 $\chi_{5} \chi_{8}$ $-$ 2 $\chi_{1} \chi_{5} \chi_{8}$ + 
    2 $\chi_{6} \chi_{8}$ + 4 $\chi_{7} \chi_{8}$ $-$ 3 $\chi_{1} \chi_{7} \chi_{8}$ $-$ 
    2 $\chi_{1}^2$ $\chi_{7} \chi_{8}$ + 3 $\chi_{2} \chi_{7} \chi_{8}$ + 
    2 $\chi_{6} \chi_{7} \chi_{8}$ $-$ 4 $\chi_{7}^2$ $\chi_{8}$ + $\chi_{1} \chi_{7}^2$ $\chi_{8}$ $-$ 
    $\chi_{8}^2$ + $\chi_{8}^3$\\ \hline
 9 & $-\chi_{1}^2$ $-$ $\chi_{1}^3$ + $\chi_{2}$ + 2 $\chi_{1} \chi_{2}$ + $\chi_{1} \chi_{3}$ + 
    $\chi_{1}^2$ $\chi_{3}$ $-$ 2 $\chi_{2} \chi_{3}$ $-$ $\chi_{4}$ $-$ $\chi_{1} \chi_{4}$ + 
    $\chi_{1}^2$ $\chi_{4}$ $-$ 2 $\chi_{2} \chi_{4}$ + 2 $\chi_{4}^2$ + 2 $\chi_{1} \chi_{5}$ + 
    $\chi_{1}^2$ $\chi_{5}$ $-$ 2 $\chi_{2} \chi_{5}$ $-$ 3 $\chi_{3} \chi_{5}$ $-$ $\chi_{5}^2$ + $\chi_{6}$ + 
    $\chi_{1} \chi_{6}$ $-$ $\chi_{1}^2$ $\chi_{6}$ + 4 $\chi_{2} \chi_{6}$ + $\chi_{1} \chi_{2} \chi_{6}$ $-$ 
    2 $\chi_{3} \chi_{6}$ $-$ $\chi_{4} \chi_{6}$ + $\chi_{1} \chi_{4} \chi_{6}$ + $\chi_{5} \chi_{6}$ $-$ 
    2 $\chi_{1} \chi_{5} \chi_{6}$ + 3 $\chi_{6}^2$ + 2 $\chi_{1} \chi_{6}^2$ + 
    $\chi_{2} \chi_{6}^2$ + 2 $\chi_{6}^3$ $-$ $\chi_{1} \chi_{7}$ + $\chi_{1}^2$ $\chi_{7}$ + 
    2 $\chi_{1}^3$ $\chi_{7}$ + 5 $\chi_{2} \chi_{7}$ $-$ 3 $\chi_{1} \chi_{2} \chi_{7}$ + 
    $\chi_{2}^2$ $\chi_{7}$ + 2 $\chi_{3} \chi_{7}$ $-$ 4 $\chi_{1} \chi_{3} \chi_{7}$ $-$ 
    $\chi_{4} \chi_{7}$ + 2 $\chi_{1} \chi_{4} \chi_{7}$ + 5 $\chi_{5} \chi_{7}$ $-$ 
    6 $\chi_{1} \chi_{5} \chi_{7}$ $-$ 2 $\chi_{1}^2$ $\chi_{5} \chi_{7}$ + 
    2 $\chi_{2} \chi_{5} \chi_{7}$ + 4 $\chi_{6} \chi_{7}$ + 5 $\chi_{1} \chi_{6} \chi_{7}$ + 
    $\chi_{1}^2$ $\chi_{6} \chi_{7}$ + 2 $\chi_{2} \chi_{6} \chi_{7}$ + 3 $\chi_{4} \chi_{6} \chi_{7}$ $-$ 
    4 $\chi_{5} \chi_{6} \chi_{7}$ + 6 $\chi_{6}^2$ $\chi_{7}$ + 2 $\chi_{1} \chi_{6}^2$ $\chi_{7}$ $-$ 
    2 $\chi_{7}^2$ + 10 $\chi_{1} \chi_{7}^2$ $-$ $\chi_{1}^2$ $\chi_{7}^2$ $-$ 
    2 $\chi_{1}^3$ $\chi_{7}^2$ $-$ 4 $\chi_{2} \chi_{7}^2$ + 4 $\chi_{1} \chi_{2} \chi_{7}^2$ $-$ 
    2 $\chi_{3} \chi_{7}^2$ + $\chi_{4} \chi_{7}^2$ $-$ 5 $\chi_{5} \chi_{7}^2$ + 
    2 $\chi_{1} \chi_{5} \chi_{7}^2$ + 2 $\chi_{6} \chi_{7}^2$ + 2 $\chi_{1} \chi_{6} \chi_{7}^2$ + 
    $\chi_{6}^2$ $\chi_{7}^2$ + 4 $\chi_{7}^3$ $-$ 7 $\chi_{1} \chi_{7}^3$ + 
    2 $\chi_{1}^2$ $\chi_{7}^3$ $-$ 3 $\chi_{6} \chi_{7}^3$ $-$ 2 $\chi_{7}^4$ $-$ $\chi_{1}^3$ $\chi_{8}$ + 
    2 $\chi_{1} \chi_{2} \chi_{8}$ $-$ $\chi_{3} \chi_{8}$ + $\chi_{4} \chi_{8}$ $-$ 
    2 $\chi_{1} \chi_{4} \chi_{8}$ + $\chi_{1} \chi_{5} \chi_{8}$ $-$ $\chi_{1}^2$ $\chi_{6} \chi_{8}$ + 
    $\chi_{2} \chi_{6} \chi_{8}$ $-$ $\chi_{5} \chi_{6} \chi_{8}$ + $\chi_{7} \chi_{8}$ $-$ 
    2 $\chi_{1} \chi_{7} \chi_{8}$ + $\chi_{4} \chi_{7} \chi_{8}$ $-$ 2 $\chi_{5} \chi_{7} \chi_{8}$ $-$ 
    $\chi_{6} \chi_{7} \chi_{8}$ $-$ 2 $\chi_{1} \chi_{6} \chi_{7} \chi_{8}$ $-$ $\chi_{7}^2$ $\chi_{8}$ + 
    $\chi_{6} \chi_{7}^2$ $\chi_{8}$ $-$ $\chi_{1} \chi_{8}^2$ + $\chi_{2} \chi_{8}^2$ + 
    $\chi_{5} \chi_{8}^2$ $-$ 2 $\chi_{7} \chi_{8}^2$ + $\chi_{1} \chi_{7} \chi_{8}^2$ + 
    $\chi_{7}^2$ $\chi_{8}^2$\\ \hline
 10 & $-$$\chi_{1}$ $-$ $\chi_{1}^2$ + $\chi_{1}^3$ + $\chi_{1}^4$ $-$ $\chi_{1} \chi_{2}$ $-$ 
    2 $\chi_{1}^2$ $\chi_{2}$ + $\chi_{3}$ $-$ $\chi_{1}^2$ $\chi_{3}$ + $\chi_{2} \chi_{3}$ $-$ $\chi_{4}$ $-$ 
    $\chi_{1} \chi_{4}$ + 2 $\chi_{1}^2$ $\chi_{4}$ + $\chi_{1}^3$ $\chi_{4}$ $-$ $\chi_{2} \chi_{4}$ $-$ 
    3 $\chi_{1} \chi_{2} \chi_{4}$ + $\chi_{3} \chi_{4}$ $-$ 2 $\chi_{4}^2$ $-$ 2 $\chi_{1} \chi_{5}$ $-$ 
    2 $\chi_{1}^2$ $\chi_{5}$ + 3 $\chi_{2} \chi_{5}$ + 2 $\chi_{3} \chi_{5}$ + 
    $\chi_{1} \chi_{4} \chi_{5}$ + 2 $\chi_{5}^2$ + $\chi_{1} \chi_{5}^2$ $-$ $\chi_{6}$ $-$ 
    $\chi_{1} \chi_{6}$ + 2 $\chi_{1}^3$ $\chi_{6}$ $-$ 3 $\chi_{1} \chi_{2} \chi_{6}$ + 
    2 $\chi_{3} \chi_{6}$ $-$ 2 $\chi_{1} \chi_{3} \chi_{6}$ $-$ 3 $\chi_{4} \chi_{6}$ $-$ 
    4 $\chi_{1} \chi_{5} \chi_{6}$ + $\chi_{2} \chi_{5} \chi_{6}$ $-$ $\chi_{5}^2$ $\chi_{6}$ + 
    2 $\chi_{1} \chi_{6}^2$ + $\chi_{1}^2$ $\chi_{6}^2$ $-$ $\chi_{2} \chi_{6}^2$ $-$ 
    $\chi_{5} \chi_{6}^2$ + $\chi_{6}^3$ + $\chi_{1} \chi_{6}^3$ $-$ 2 $\chi_{7}$ + 
    2 $\chi_{1}^2$ $\chi_{7}$ $-$ $\chi_{1}^3$ $\chi_{7}$ $-$ $\chi_{1}^4$ $\chi_{7}$ $-$ $\chi_{2} \chi_{7}$ $-$ 
    $\chi_{1} \chi_{2} \chi_{7}$ + 3 $\chi_{1}^2$ $\chi_{2} \chi_{7}$ $-$ $\chi_{2}^2$ $\chi_{7}$ + 
    3 $\chi_{3} \chi_{7}$ $-$ 2 $\chi_{1} \chi_{3} \chi_{7}$ $-$ 9 $\chi_{4} \chi_{7}$ $-$ 
    $\chi_{1}^2$ $\chi_{4} \chi_{7}$ $-$ 2 $\chi_{2} \chi_{4} \chi_{7}$ + 2 $\chi_{1} \chi_{5} \chi_{7}$ $-$ 
    $\chi_{2} \chi_{5} \chi_{7}$ + 3 $\chi_{4} \chi_{5} \chi_{7}$ $-$ $\chi_{5}^2$ $\chi_{7}$ $-$ 
   3 $\chi_{6} \chi_{7}$ + 8 $\chi_{1} \chi_{6}
   \chi_{7}$ + $\chi_{1}^2$ $\chi_{6} \chi_{7}$
   $-$
    $\chi_{1}^3$ $\chi_{6} \chi_{7}$ $-$ $\chi_{2} \chi_{6} \chi_{7}$ + 
    2 $\chi_{1} \chi_{2} \chi_{6} \chi_{7}$ $-$ 3 $\chi_{3} \chi_{6} \chi_{7}$ $-$ 
    6 $\chi_{4} \chi_{6} \chi_{7}$ $-$ $\chi_{5} \chi_{6} \chi_{7}$ $-$ 
    2 $\chi_{1} \chi_{5} \chi_{6} \chi_{7}$ + 3 $\chi_{6}^2$ $\chi_{7}$ + 
    5 $\chi_{1} \chi_{6}^2$ $\chi_{7}$ + 2 $\chi_{6}^3$ $\chi_{7}$ $-$ 5 $\chi_{7}^2$ + 
    4 $\chi_{1} \chi_{7}^2$ $-$ 2 $\chi_{1}^2$ $\chi_{7}^2$ + $\chi_{1}^3$ $\chi_{7}^2$ + 
    3 $\chi_{2} \chi_{7}^2$ $-$ 5 $\chi_{4} \chi_{7}^2$ + 2 $\chi_{1} \chi_{4} \chi_{7}^2$ + 
    9 $\chi_{5} \chi_{7}^2$ $-$ 4 $\chi_{1} \chi_{5} \chi_{7}^2$ + $\chi_{1}^2$ $\chi_{6} \chi_{7}^2$ + 
    $\chi_{2} \chi_{6} \chi_{7}^2$ $-$ 2 $\chi_{5} \chi_{6} \chi_{7}^2$ + 2 $\chi_{6}^2$ $\chi_{7}^2$ $-$ 
    $\chi_{7}^3$ + 2 $\chi_{1} \chi_{7}^3$ $-$ $\chi_{2} \chi_{7}^3$ + 4 $\chi_{4} \chi_{7}^3$ $-$ 
    $\chi_{5} \chi_{7}^3$ $-$ 4 $\chi_{6} \chi_{7}^3$ + 4 $\chi_{7}^4$ $-$ 3 $\chi_{1} \chi_{7}^4$ $-$ 
    $\chi_{6} \chi_{7}^4$ $-$ $\chi_{8}$ $-$ 3 $\chi_{1} \chi_{8}$ $-$ $\chi_{1}^2$ $\chi_{8}$ + 
    $\chi_{2} \chi_{8}$ $-$ 2 $\chi_{1} \chi_{2} \chi_{8}$ + 3 $\chi_{3} \chi_{8}$ + 
    $\chi_{1} \chi_{3} \chi_{8}$ $-$ 3 $\chi_{4} \chi_{8}$ + $\chi_{1} \chi_{4} \chi_{8}$ + 
    2 $\chi_{5} \chi_{8}$ + $\chi_{1} \chi_{5} \chi_{8}$ $-$ $\chi_{1}^2$ $\chi_{5} \chi_{8}$ + 
    2 $\chi_{2} \chi_{5} \chi_{8}$ + $\chi_{5}^2$ $\chi_{8}$ $-$ 2 $\chi_{6} \chi_{8}$ $-$ 
    4 $\chi_{1} \chi_{6} \chi_{8}$ $-$ $\chi_{1}^2$ $\chi_{6} \chi_{8}$ $-$ $\chi_{4} \chi_{6} \chi_{8}$ $-$ 
    $\chi_{5} \chi_{6} \chi_{8}$ $-$ 2 $\chi_{1} \chi_{6}^2$ $\chi_{8}$ $-$ 5 $\chi_{7} \chi_{8}$ + 
    2 $\chi_{1} \chi_{7} \chi_{8}$ $-$ $\chi_{2} \chi_{7} \chi_{8}$ + 
    $\chi_{1} \chi_{2} \chi_{7} \chi_{8}$ $-$ $\chi_{3} \chi_{7} \chi_{8}$ $-$ $\chi_{4} \chi_{7} \chi_{8}$ + 
    2 $\chi_{5} \chi_{7} \chi_{8}$ + $\chi_{1} \chi_{5} \chi_{7} \chi_{8}$ $-$ 
    3 $\chi_{6} \chi_{7} \chi_{8}$ $-$ $\chi_{1} \chi_{6} \chi_{7} \chi_{8}$ + 2 $\chi_{7}^2$ $\chi_{8}$ $-$ 
    2 $\chi_{1}^2$ $\chi_{7}^2$ $\chi_{8}$ + 2 $\chi_{2} \chi_{7}^2$ $\chi_{8}$ + 
    3 $\chi_{5} \chi_{7}^2$ $\chi_{8}$ $-$ 4 $\chi_{6} \chi_{7}^2$ $\chi_{8}$ $-$ 2 $\chi_{7}^3$ $\chi_{8}$ + 
    2 $\chi_{1} \chi_{7}^3$ $\chi_{8}$ + 2 $\chi_{7}^4$ $\chi_{8}$ $-$ $\chi_{8}^2$ + 
    $\chi_{1} \chi_{8}^2$ $-$ $\chi_{4} \chi_{8}^2$ $-$ $\chi_{6} \chi_{8}^2$ + 
    $\chi_{1} \chi_{6} \chi_{8}^2$ $-$ $\chi_{7} \chi_{8}^2$ + 2 $\chi_{1} \chi_{7} \chi_{8}^2$ $-$ 
    $\chi_{7}^2$ $\chi_{8}^2$ $-$ $\chi_{8}^3$ \\ \hline
    11 &
3 $\chi _7^6$ $+$ $3 \chi _7^5$ $+$ $5 \chi _5 \chi _7^4$ $-$ $15 \chi _6 \chi
_7^4$ $+$ $2 \chi _8 \chi _7^4$ $-$ $8 \chi _7^4$ $-$ $2 \chi _1^2 \chi _7^3$
$-$ $5 \chi _6^2 \chi _7^3$ $-$ $\chi _8^2 \chi _7^3$ $+$ $\chi _1 \chi _7^3$
$+$ $2 \chi _2 \chi    _7^3$ $+$ $4 \chi _3 \chi _7^3$ $-$ $5 \chi _4 \chi
_7^3$ $+$ $2 \chi _1 \chi _5 \chi _7^3$ $+$ $14 \chi _5 \chi _7^3$ $-$ $6 \chi
_1 \chi _6 \chi _7^3$ $-$ $10 \chi _6 \chi _7^3$ $+$ $2 \chi _1 \chi _8 \chi_7^3$ $+$ $\chi _6 \chi _8 \chi _7^3$ $+$ $4 \chi _8 \chi _7^3$ $-$ $10 \chi _7^3$ $+$ $2 \chi _5^2 \chi _7^2$ $+$ $2 \chi _1 \chi _6^2 \chi _7^2$ $+$ $12 \chi _6^2 \chi _7^2$ $-$ $\chi _1 \chi _8^2 \chi _7^2$ $-$ $9 \chi _1 \chi _7^2$ $+$ $2 \chi _2 \chi _7^2$ $+$ $5 \chi _3 \chi
   _7^2$ $-$ $\chi _1 \chi _4 \chi _7^2$ $-$ $3 \chi _4 \chi _7^2$ $-$ $\chi _1^2 \chi _5 \chi _7^2$ $+$ $7 \chi _1 \chi _5 \chi _7^2$ $+$ $2 \chi _2 \chi _5 \chi _7^2$ $-$ $5 \chi _5 \chi _7^2$ $+$ $3 \chi _4 \chi _6 \chi _7^2$ $-$ $13 \chi _5
   \chi _6 \chi _7^2$ $+$ $17 \chi _6 \chi _7^2$ $+$ $3 \chi _1^2 \chi _8 \chi _7^2$ $+$ $\chi _1 \chi _8 \chi _7^2$ $-$ $2 \chi _2 \chi _8 \chi _7^2$ $+$ $\chi _4 \chi _8 \chi _7^2$ $+$ $2 \chi _5 \chi _8 \chi _7^2$ $+$ $3 \chi _1 \chi _6
   \chi _8 \chi _7^2$ $-$ $11 \chi _6 \chi _8 \chi _7^2$ $-$ $7 \chi _8 \chi _7^2$ $+$ $9 \chi _7^2$ $-$ $3 \chi _1^3 \chi _7$ $+$ $13 \chi _6^3 \chi _7$ $+$ $2 \chi _1^2 \chi _7$ $+$ $4 \chi _1 \chi _5^2 \chi _7$ $+$ $8 \chi _5^2 \chi _7$ $-$ $\chi
   _1^2 \chi _6^2 \chi _7$ $+$ $9 \chi _1 \chi _6^2 \chi _7$ $+$ $\chi _2 \chi _6^2 \chi _7$ $-$ $5 \chi _5 \chi _6^2 \chi _7$ $+$ $24 \chi _6^2 \chi _7$ $+$ $\chi _1^2 \chi _8^2 \chi _7$ $-$ $2 \chi _1 \chi _8^2 \chi _7$ $-$ $\chi _2 \chi
   _8^2 \chi _7$ $-$ $\chi _5 \chi _8^2 \chi _7$ $+$ $\chi _6 \chi _8^2 \chi _7$ $+$ $3 \chi _8^2 \chi _7$ $+$ $9 \chi _1 \chi _7$ $+$ $3 \chi _1 \chi _2 \chi _7$ $-$ $7 \chi _2 \chi _7$ $+$ $2 \chi _1 \chi _3 \chi _7$ $-$ $8 \chi _3 \chi _7$ $+$ $5
   \chi _1^2 \chi _4 \chi _7$ $-$ $2 \chi _1 \chi _4 \chi _7$ $-$ $8 \chi _2 \chi _4 \chi _7$ $+$ $3 \chi _4 \chi _7$ $-$ $8 \chi _1 \chi _5 \chi _7$ $+$ $\chi _2 \chi _5 \chi _7$ $+$ $5 \chi _3 \chi _5 \chi _7$ $-$ $8 \chi _4 \chi _5 \chi
   _7$ $-$ $19 \chi _5 \chi _7$ $+$ $2 \chi _1^2 \chi _6 \chi _7$ $+$ $13 \chi _1 \chi _6 \chi _7$ $-$ $3 \chi _2 \chi _6 \chi _7$ $-$ $10 \chi _3 \chi _6 \chi _7$ $-$ $15 \chi _1 \chi _5 \chi _6 \chi _7$ $-$ $21 \chi _5 \chi _6 \chi _7$ $+$ $14
   \chi _6 \chi _7$ $-$ $\chi _1^3 \chi _8 \chi _7$ $-$ $\chi _1^2 \chi _8 \chi _7$ $-$ $4 \chi _6^2 \chi _8 \chi _7$ $-$ $8 \chi _1 \chi _8 \chi _7$ $+$ $\chi _1 \chi _2 \chi _8 \chi _7$ $-$ $2 \chi _2 \chi _8 \chi _7$ $+$ $4 \chi _3 \chi
   _8 \chi _7$ $-$ $3 \chi _1 \chi _4 \chi _8 \chi _7$ $-$ $3 \chi _4 \chi _8 \chi _7$ $+$ $6 \chi _1 \chi _5 \chi _8 \chi _7$ $+$ $2 \chi _5 \chi _8 \chi _7$ $-$ $2 \chi _1^2 \chi _6 \chi _8 \chi _7$ $-$ $10 \chi _1 \chi _6 \chi _8
   \chi _7$ $+$ $2 \chi _2 \chi _6 \chi _8 \chi _7$ $+$ $2 \chi _5 \chi _6 \chi _8 \chi _7$ $-$ $5 \chi _6 \chi _8 \chi _7$ $-$ $2 \chi _8 \chi _7$ $+$ $2 \chi _7$ $+$ $\chi _1^4$ $+$ $3 \chi _6^4$ $+$ $3 \chi _1^3$ $+$ $5 \chi _1 \chi _6^3$ $+$ $7 \chi
   _6^3$ $-$ $\chi _1 \chi _8^3$ $+$ $\chi _6 \chi _8^3$ $+$ $\chi _8^3$ $+$ $3 \chi _1^2$ $+$ $2 \chi _2^2$ $+$ $\chi _3^2$ $-$ $\chi _1 \chi _4^2$ $+$ $5 \chi _1 \chi _5^2$ $+$ $4 \chi _2 \chi _5^2$ $+$ $9 \chi _5^2$ $+$ $3 \chi _1^2 \chi _6^2$ $+$ $10 \chi _1 \chi
   _6^2$ $+$ $\chi _1 \chi _2 \chi _6^2$ $-$ $6 \chi _3 \chi _6^2$ $-$ $4 \chi _1 \chi _5 \chi _6^2$ $-$ $4 \chi _5 \chi _6^2$ $+$ $6 \chi _6^2$ $+$ $2 \chi _1^2 \chi _8^2$ $+$ $3 \chi _1 \chi _8^2$ $-$ $3 \chi _2 \chi _8^2$ $+$ $2 \chi _4 \chi
   _8^2$ $-$ $\chi _1 \chi _5 \chi _8^2$ $-$ $4 \chi _5 \chi _8^2$ $+$ $\chi _1 \chi _6 \chi _8^2$ $+$ $\chi _6 \chi _8^2$ $+$ $\chi _8^2$ $+$ $\chi _1$ $-$ $2 \chi _1^2 \chi _2$ $-$ $3 \chi _1 \chi _2$ $-$ $3 \chi _1^2 \chi _3$ $-$ $5 \chi _1 \chi _3$ $+$ $3 \chi
   _2 \chi _3$ $-$ $\chi _3$ $+$ $\chi _1^2 \chi _4$ $+$ $2 \chi _1 \chi _4$ $-$ $2 \chi _1 \chi _2 \chi _4$ $-$ $\chi _2 \chi _4$ $-$ $\chi _3 \chi _4$ $+$ $\chi _1^4 \chi _5$ $-$ $2 \chi _1^3 \chi _5$ $-$ $10 \chi _1^2 \chi _5$ $+$ $2 \chi _2^2 \chi _5$ $-$ $7
   \chi _1 \chi _5$ $-$ $4 \chi _1^2 \chi _2 \chi _5$ $+$ $2 \chi _1 \chi _2 \chi _5$ $+$ $8 \chi _2 \chi _5$ $+$ $4 \chi _1 \chi _3 \chi _5$ $+$ $7 \chi _3 \chi _5$ $-$ $4 \chi _4 \chi _5$ $-$ $\chi _5$ $+$ $2 \chi _1^3 \chi _6$ $+$ $7 \chi _1^2 \chi
   _6$ $+$ $\chi _2^2 \chi _6$ $+$ $\chi _5^2 \chi _6$ $+$ $7 \chi _1 \chi _6$ $-$ $\chi _1 \chi _2 \chi _6$ $-$ $3 \chi _2 \chi _6$ $-$ $5 \chi _1 \chi _3 \chi _6$ $-$ $6 \chi _3 \chi _6$ $+$ $\chi _1^2 \chi _4 \chi _6$ $-$ $\chi _1 \chi _4 \chi _6$ $-$ $4
   \chi _2 \chi _4 \chi _6$ $+$ $\chi _4 \chi _6$ $-$ $4 \chi _1^2 \chi _5 \chi _6$ $-$ $9 \chi _1 \chi _5 \chi _6$ $+$ $\chi _2 \chi _5 \chi _6$ $+$ $\chi _4 \chi _5 \chi _6$ $-$ $8 \chi _5 \chi _6$ $+$ $2 \chi _6$ $-$ $\chi _1^3 \chi _8$ $-$ $\chi
   _1^2 \chi _8$ $-$ $2 \chi _5^2 \chi _8$ $-$ $\chi _1 \chi _6^2 \chi _8$ $+$
   $\chi _1 \chi _8$ $+$ $\chi _1 \chi _2 \chi _8$ $+$ $2 \chi _1 \chi _3 \chi
   _8$ $+$ $\chi _1^2 \chi _4 \chi _8$ $-$ $2 \chi _2 \chi _4 \chi _8$ $+$ $2
   \chi _1^2 \chi _5 \chi _8$ $+$ $3 \chi _1 \chi _5 \chi _8$ $+$ $2 \chi _5
   \chi _8$ $-$ $\chi _1^3 \chi _6 \chi _8$ $-$ $\chi _1 \chi _6 \chi _8$ $+$
   $3 \chi _1 \chi _2 \chi _6 \chi _8$ $-$ $\chi _2 \chi _6 \chi _8$ $-$ $3
   \chi _3 \chi _6 \chi _8$ $+$ $\chi _4 \chi _6 \chi _8$\\
   \hline
     \end{longtable}
\newpage

\end{appendix}

\bibliography{miabiblio}
\end{document}